\documentclass[reqno]{amsart}

\usepackage[utf8]{inputenc}
\usepackage{array}
\newcolumntype{L}{>{$}l<{$}}
\newcolumntype{C}{>{$}c<{$}}

\usepackage{amssymb, amsmath, amsthm, esint}
\usepackage{bm} 

\usepackage[dvips]{graphicx}
\usepackage[usenames]{xcolor}

\newcommand{\blue}[1]{\textcolor{blue}{#1}}

\newcommand{\alx}[1]{\textcolor{teal}{#1}}

\usepackage[normalem]{ulem}
\newcommand{\stkout}[1]{\ifmmode\text{\sout{\ensuremath{#1}}}\else\sout{#1}\fi}

\usepackage{marginnote}

\DeclareMathOperator\arctanh{arctanh}
\DeclareMathOperator\arcoth{arcoth}
\DeclareMathOperator\arcot{arcot}

\newtheorem{theorem}{Theorem}
\newtheorem{proposition}[theorem]{Proposition}

\newtheorem{rem}[theorem]{Remark}
\newtheorem{lemma}[theorem]{Lemma}
\newtheorem{corollary}[theorem]{Corollary}
\newtheorem{cor}[theorem]{Corollary}
\newtheorem{example}[theorem]{Example}

\newcommand{\R}{\mathbb R}
\newcommand{\N}{\mathbb N}
\newcommand{\Sp}{\mathbb S}
\newcommand{\C}{\mathcal{C}}

\newcommand{\la}{\langle}
\newcommand{\ra}{\rangle}

\newcommand{\n}{{\bf n}}
\newcommand{\tr}{{\rm tr}}

\newcommand{\grad}{{\nabla }}
\newcommand{\dist}{{\rm dist}}

\newcommand{\tf}{{\tilde f}}
\newcommand{\roll}{{\rm roll}}

\newcommand{\cut}{{\rm cut}}
\newcommand{\Ric}{{\rm Ric}}

\newcommand{\nd}{\frac{\partial}{\partial\n}}

\newcommand{\hb}{{\bar{h}}}
\newcommand{\htd}{\tilde{h}}
\newcommand{\divergence}{\mbox{div}}

\let\oldnu\nu
\renewcommand{\nu}{{\bm\oldnu}}

\newcommand{\constA}{{\bf A}} 
\newcommand{\constB}{{\bf B}} 
\newcommand{\cbm}{\overline{\constA}}
\newcommand{\scbm}{\overline{\constB}}

\title[Relation between the Steklov and the Laplacian spectrum]{The Steklov and Laplacian spectra of Riemannian manifolds with boundary\\
}
\author{Bruno Colbois}
\address{Universit\'e de Neuch\^atel, Institut de Math\'ematiques, Rue
	Emile-Argand 11, CH-2000 Neuch\^atel, Switzerland}
\email{bruno.colbois@unine.ch}
\author{Alexandre Girouard}
\address{D\'e\-par\-te\-ment de math\'ematiques et de
	sta\-tistique, Univer\-sit\'e Laval, Pavillon Alexandre\-Vachon,
	1045, av. de la M\'edecine,
	Qu\'ebec Qc G1V 0A6, 
	Canada }
\email{alexandre.girouard@mat.ulaval.ca}
\author{Asma Hassannezhad}
\address{University of Bristol, School of Mathematics, University Walk, Bristol BS8 1TW, UK}
\email{asma.hassannezhad@bristol.ac.uk}



\begin{document}

\begin{abstract}
Given two compact Riemannian manifolds $M_1$ and $M_2$ such that their respective boundaries $\Sigma_1$ and $\Sigma_2$ admit neighbourhoods $\Omega_1$ and $\Omega_2$ which are isometric, we prove the existence of a constant $C$ such that $|\sigma_k(M_1)-\sigma_k(M_2)|\leq C$ for each $k\in\N$.  The constant $C$ depends only on the geometry of $\Omega_1\cong\Omega_2$. This follows from a quantitative relationship between the Steklov eigenvalues $\sigma_k$ of a
  compact Riemannian manifold $M$ and the eigenvalues $\lambda_k$ of the Laplacian on
  its boundary. Our main result states that the difference $|\sigma_k-\sqrt{\lambda_k}|$ is bounded above by a constant which depends on the geometry of $M$ only in a neighbourhood of its boundary. 
  The proofs are based on a Pohozaev identity and on comparison geometry for principal curvatures of parallel hypersurfaces. In several situations, the constant $C$ is given explicitly in terms of bounds on the geometry of $\Omega_1\cong\Omega_2$.
\end{abstract}

\maketitle
%
%
%
%
\setlength{\parskip}{.3cm}


\section{Introduction}
Let $M$ be a smooth compact Riemannian manifold of dimension $n+1\geq 2$,
with nonempty boundary $\Sigma$. The Steklov eigenvalue problem on $M$ is to find all numbers $\sigma\in\R$ for which there exists a nonzero function $u\in C^\infty(M)$ which satisfies
\begin{equation*}
\begin{cases}
\Delta u=0 & \mbox{ in }M,\\
\frac{\partial u}{\partial \n}=\sigma u & \mbox{ on }\Sigma.
\end{cases}
\end{equation*}
Here $\n$ is the  outward unit normal along $\Sigma$ and $\Delta=\mbox{div}\circ\nabla$ is the Laplace-Beltrami operator acting on $C^\infty(M)$. The Steklov problem  has a discrete unbounded spectrum
\begin{equation*}
0=\sigma_1\le\sigma_2\leq\sigma_3\leq\cdots\nearrow+\infty,
\end{equation*}
where each eigenvalue is repeated according to its multiplicity.
For background on this problem, see \cite{GP15,legacy} and references therein.

\subsection{The Dirichlet-to-Neumann map and spectral asymptotics}
Let $\mathcal{H}:C^\infty(\Sigma)\rightarrow C^\infty(M)$ be the harmonic extension operator: the function $u=\mathcal{H}f$ satisfies $u=f$ on $\Sigma$ and $\Delta u=0$ in $M$.
The Steklov eigenvalues of $M$ are the eigenvalues of the Dirichlet-to-Neumann (DtN) map $D:C^\infty(\Sigma)\rightarrow C^\infty(\Sigma)$, which is defined by
$$Df=\nd\mathcal{H}{f}.$$
The DtN map is a first order self-adjoint elliptic pseudodifferential operator \cite[pp. 37--38]{Taylor}. Its principal symbol is given by $p(x,\xi)=|\xi|$, while that of the Laplace operator $\Delta_\Sigma:C^\infty(\Sigma)\rightarrow C^\infty(\Sigma)$ is $|\xi|^2$. Standard elliptic theory \cite{hormander,shubin} then implies that
\begin{gather}\label{asymptotics:weak}
\sigma_k\sim\sqrt{\lambda_k}\sim 2\pi\left(\frac{k}{\omega_n\mbox{Vol}_n(\Sigma)}\right)^{1/n}\qquad\mbox{ as }\quad k\to\infty,
\end{gather}
where $\omega_n$ is the volume of the unit ball $B(0,1)\subset\R^n$.
It follows that manifolds $M_1$ and $M_2$ which have isometric boundaries satisfy
$\sigma_k(M_1)\sim\sigma_k(M_2)$ as $k\to\infty$.

It was proved in \cite{LeeUhlmann} that the full symbol of the DtN map is determined by the Taylor series of the Riemannian metric of $M$ in the normal direction along its boundary $\Sigma$ (see also \cite{PoltSher}). 
Consider a closed Riemannian manifold $(\Sigma,g_\Sigma)$ and two compact Riemannian manifolds $(M_1,g_1)$ and $(M_2,g_2)$ with the same boundary $\Sigma=\partial M_1=\partial M_2$ such that $g_1\bigl|\bigr._{\Sigma}=g_2\bigl|\bigr._{\Sigma}=g_\Sigma$. If the metrics $g_1$ and $g_2$ have the same Taylor series on $\Sigma$, then $M_1$ and $M_2$
have asymptotically equivalent\footnote{The notation $O(k^{-\infty})$ designates a quantity which tends to zero faster than any power of $k$.}
Steklov spectra:
\begin{gather}\label{asymptotics:strong}
\sigma_k(M_1)=\sigma_k(M_2)+O(k^{-\infty})\quad\mbox{ as }\quad k\to\infty.
\end{gather}
See \cite[Lemma 2.1]{GPPS}.
In particular, $\lim_{k\to\infty}\sigma_k(M_1,g_1)-\sigma_k(M_2,g_2)=0$. 

\subsection{Main results}
The asymptotic behavior described by \eqref{asymptotics:weak} or \eqref{asymptotics:strong} does not contain any information regarding an individual eigenvalue $\sigma_k$. In order to obtain such information, stronger hypothesis are needed.
Indeed on any smooth compact Riemannian manifold $(M,g_0)$ with boundary, there exists a family of Riemannian metrics $(g_\epsilon)_{\epsilon\in\R_+}$ such that $g_\epsilon=g_0$ on a neighbourhood $\Omega_\epsilon$ of $\Sigma=\partial M$, while for each $k\in\N$,
$$\lim_{\epsilon\to\infty}\sigma_k(M,g_\epsilon)=0.$$
See \cite{CEG17} and \cite{CGM} for two such constructions.
If $n\geq 2$, there also exists a family of Riemannian metrics such that $g_\epsilon=g_0$ on a neighbourhood $\Omega_\epsilon$ of $\Sigma$ and 
$$\lim_{\epsilon\to\infty}\sigma_2(M,g_\epsilon)=+\infty.$$
See \cite{CEG17} for the construction of these families.
In this last construction, the neighbourhoods $\Omega_\epsilon$ are shrinking to $\Sigma$ as $\epsilon\to\infty$:
$$\bigcap_{\epsilon\in (0,\infty)}\Omega_\epsilon=\Sigma.$$
However, if the manifolds $M_1$ and $M_2$ are uniformly isometric near their boundary, the situation is completely different. 

\begin{theorem}\label{thm:mainstek}
  Given two compact Riemannian manifolds with boundary $M_1$ and $M_2$ such that their respective boundaries $\Sigma_1$ and $\Sigma_2$ admit neighbourhoods $\Omega_1$ and $\Omega_2$ which are isometric, there exists a constant $C$, which depends explicitly on the geometry of $\Omega_1\cong\Omega_2$, such that $|\sigma_k(M_1)-\sigma_k(M_2)|\leq C$ for each $k\in\N$. 
\end{theorem}

Not only the difference $|\sigma_k(M_1)-\sigma_k(M_2)|$ is bounded for each given $k\in\N$, but the bound is uniform: the constant $C$ does not depend on $k$. The existence of $C$ could be obtained from a Dirichlet-to-Neumann bracketting argument, but this would not lead to an explicit expression.

Theorem \ref{thm:mainstek} is a manifestation of the the principle stating that Steklov eigenvalues are mostly sensitive to the geometry of a manifold near its boundary. The more information we have on the metric $g$ on, and near, the boundary $\Sigma$, the more we can say about the Steklov spectrum of $M$. Knowing the metric on $\Sigma$ leads to the asymptotic formula \eqref{asymptotics:weak}, then knowing its Taylor series in the normal direction along $\Sigma$ leads to the refined asymptotic formula \eqref{asymptotics:strong}, and finally the knowledge of the metric on a neighbourhood of the boundary provides the extra information needed to control individual eigenvalues in Theorem \ref{thm:mainstek}.

\begin{rem}
Let $\ell$ be the number of connected components of the boundary $\Sigma$. Under the hypothesis of Theorem \ref{thm:mainstek}, it is well known that for $k \geq \ell$ the following holds:
\begin{gather*} 
C^{-1}\leq\frac{\sigma_k(M_1)}{\sigma_k(M_2)}\leq C.
\end{gather*}
This was used for instance in \cite{CEG17,CGR}. 
Theorem \ref{thm:mainstek} is stronger since it implies
$$\frac{1}{1+\frac{C}{\sigma_k(M_1)}}\leq\frac{\sigma_k(M_1)}{\sigma_k(M_2)}\leq1+\frac{C}{\sigma_k(M_2)}.$$
It is also known that $\sigma_k(M_i)\geq\sigma_k^N\geq\sigma_{\ell}^N$, where
$\sigma_k^N$ refer to the mixed Neumann-Steklov problem on $\Omega_1\cong\Omega_2$. See \cite{CEG17} for details.
\end{rem}

Theorem \ref{thm:mainstek} follows from our main result, which is
a quantitative comparison between the Steklov eigenvalues of $M$ and the eigenvalues of the Laplace operator $\Delta_\Sigma$ on its boundary $\Sigma$, which are denoted
$0=\lambda_1\leq\lambda_2\leq\cdots\nearrow\infty$.
	Let $n\in\N$ and let $\alpha,\beta,\kappa_-,\kappa_+\in\R$ and $h>0$ be such that $\alpha\leq\beta$ and $\kappa_-\leq\kappa_+$. Consider the class $\mathcal{M}=\mathcal{M}(n,\alpha,\beta,\kappa_-,\kappa_+,h)$ of smooth compact Riemannian manifolds $M$ of dimension $n+1$ with nonempty boundary $\Sigma$, which satisfy the following hypotheses:
\begin{itemize}
	\item[(H1)] The rolling radius\footnote{The rolling radius will be defined in Section \ref{section:prelim}.} of $M$ satisfies $\hb:=\roll(M)\geq h$.
	\item[(H2)] The sectional curvature $K$ satisfies
	$\alpha\leq K\leq \beta$ on the tubular neighbourhood
	$$M_\hb=\{x\in M\,:\,d(x,\Sigma)<\hb\}.$$
	\item[(H3)] The principal curvatures of the boundary $\Sigma$ satisfy
	$\kappa_-\leq\kappa_i\leq\kappa_+.$ 
\end{itemize}
The main result of this paper is the following.
\begin{theorem}\label{mainintro}
There exist explicit constants $\constA=\constA(n,\alpha,\beta,\kappa_-,\kappa_+,h)$ and $\constB=\constB(n,\alpha,\kappa_-,h)$ such that each manifold $M$ in the class $\mathcal{M}$ satisfies the following inequalities for each~$k\in\N$,
\begin{gather}
\lambda_k\leq \sigma_k^2+\constA\sigma_k,\label{ineq:main1}\\
\sigma_k\leq \constB+\sqrt{\constB^2+\lambda_k}.\label{ineq:thm:mainstek}
\end{gather}
In particular, for each $k\in\N$,
$|\sigma_k-\sqrt{\lambda_k}|<\max\{\constA,2\constB\}.$
\end{theorem}

Under the hypotheses of Theorem \ref{mainintro} one can take
$$\constA=\frac{n}{\htd}+\sqrt{|\alpha|+\kappa_-^2}
\qquad\mbox{ and }\qquad
\constB=\frac{1}{2\hb}+\frac{n}{2}\sqrt{|\alpha|+\kappa_-^2},$$
where $\tilde h\leq \hb$ is a positive constant depending on $\bar h$, $\beta$, and $\kappa_+$. See \eqref{def:htd} for the precise definition.
Various other values for the constants $\constA$ and $\constB$ will be given under more restrictive assumptions in Section~\ref{section:introsigned}.

This theorem is in the spirit of the very nice result of \cite{PS16} where a similar statement is proved for Euclidean domains. 
	The dependence of $\constA$ on the dimension $n$ is necessary. Indeed, it was observed in \cite{PS16} that on the ball $B(0,R)\subset\R^{n+1}$ the following holds for each $k$: $$\lambda_k=\sigma_k^2+\frac{n-1}{R}\sigma_k.$$ 
	We do not know if the dependence of $\constB$ on $n$ is necessary in general. However, see Theorem \ref{thm:signedcurvature} for situations where $\constB$ does not depend on $n$.

        Inequality \eqref{ineq:thm:mainstek} also holds under weaker hypotheses on $M$. 
Let $\mathcal{M'}=\mathcal{M'}(n,\alpha,\kappa_-,h)$ be the class of smooth compact Riemannian manifolds of {dimension $n+1$} with nonempty boundary $\Sigma$, which satisfy (H1) and
\begin{itemize}
			\item[(H2$'$)] The Ricci curvature satisfies $\mbox{Ric}\geq n\alpha$ on $M_{\hb}$.
			\item[(H3$'$)] The mean curvature of $\Sigma$ satisfies $H\geq \kappa_-.$
		\end{itemize}
		Then we have the following theorem. 
\begin{theorem}\label{rem:weakerH}
There exists an explicit constant $\constB=\constB(n,\alpha,\kappa_-,h)$ such that each manifold $M$ in the class $\mathcal{M'}$ satisfies the following inequality for each~$k\in\N$,
\begin{gather}
\sigma_k\leq \constB+\sqrt{\constB^2+\lambda_k}.\label{ineq:thm:mainstek2}
\end{gather}
\end{theorem}

\subsection{Discussion and previous results}
Quantitative estimates relating individual Steklov eigenvalues to eigenvalues of the tangential Laplacian $\Delta_\Sigma$ have been studied in \cite{WX,CGR,CEG17,CGG, Kar15,YangYu,PS16,Xi}.  They are relatively easy to obtain if the manifold $M$ is isometric (or quasi-isometric with some control) to a product near its boundary. See for example
\cite[Lemma 2.1]{CGG}. In this context however, it is usually the quotient $\sqrt{\lambda_k}/\sigma_k$ which is controlled.

Theorem \ref{mainintro} is a generalization of the fundamental result of \cite{PS16}, and also of \cite{Xi}. For bounded Euclidean domains $\Omega\subset\R^{n+1}$ with smooth connected boundary $\Sigma=\partial\Omega$, 
Provenzano and Stubbe \cite{PS16} proved a comparison result similar to Theorem \ref{mainintro}, with the constants $\constA$ and $\constB$ replaced by a constant  ${\bf C}_{\Omega}$ depending  on the dimension, the
maximum of the  mean of the absolute values of the principal
curvatures on $\Sigma=\partial\Omega$ and the rolling
radius of $\Omega$. Their main insight was to use a generalized Pohozaev identity\footnote{See Section \ref{section:pohozaev} for details on the Pohozaev identity.} in order to compare the Dirichlet energy of an harmonic function $u\in C^{\infty}(\overline{\Omega})$ to the $L^2$-norm of its normal derivative along $\Sigma$. 
In \cite{Xi}, Xiong  extended the results of \cite{PS16} to the Riemannian setting under rather stringent hypotheses. Indeed he considered  domains $\Omega$ with convex boundary in a complete Riemannian manifold $X$ with either nonpositive ($K_X\leq 0$) or strictly positive ($K_X>0$) sectional curvature, with some hypotheses on the shape operator (or second fundamental form). These condition in particular implies that the boundary is connected.

Theorem \ref{mainintro} improves these results in several ways. First we consider compact Riemannian manifolds with boundary. This is more general than bounded domains in a complete manifold (see \cite{PV16} for a discussion of this question). Another strength of Theorem \ref{mainintro} is that we require geometric control of the manifold $M$ only in the neighbourhood $M_\hb$ of its boundary. Finally, we are neither assuming the boundary $\Sigma$ to be connected nor the sectional curvature to have a constant sign.
\begin{rem}
	Let $\ell$ be the number of connected components of $\Sigma$. Then $\lambda_i=0,$ for $i=1,\ldots,\ell$. Thus, inequality \eqref{ineq:main1} is of interest only  for $k\ge \ell+1$. 
	Furthermore inequality \eqref{ineq:thm:mainstek} becomes $\sigma_k\le 2\constB$ for $k=1,\ldots,\ell$. There are indeed examples where $\sigma_\ell$ is arbitrarily small. See Section \ref{section:examples}.
\end{rem}

\begin{rem}
	The \emph{inner and exterior boundary} of a domain $D\subset M$ are defined to be
		$$\partial_ID:=\partial D\cap \mbox{int}\, M,\qquad
		\partial_ED:=\overline{D}\cap\Sigma.
		$$
The Cheeger and Jammes constant of $M$ are defined to be
$$h_M=\inf_D\frac{\mbox{Vol}_n(\partial_ID)}
{\mbox{Vol}_{n+1}(D)}\qquad\mbox{ and }\qquad
h_M'=\inf_D\frac{\mbox{Vol}_n(\partial_ID)}
{\mbox{Vol}_{n}(\partial_ED)},$$
where both infima are taken over all domains $D\subset M$ such that
$\mbox{Vol}_{n+1}(D)\leq\frac{\mbox{Vol}_{n+1}(M)}{2}$.
In \cite{jammes} Jammes proved that 
	\begin{gather}\label{ineq:pjammes}
	\sigma_2\geq \frac{1}{4}h_Mh_M'.
	\end{gather}
	Another useful lower bound can be deduced from Theorem \ref{mainintro} and the Cheeger inequality for $\Sigma$, which states that
	\begin{gather}\label{ineq:cheeger}
	\lambda_2\geq\frac{1}{4}h_\Sigma^2,
	\end{gather}
        where $h_\Sigma$ is the Cheeger constant of the closed manifold $\Sigma$.
	 Indeed, it follows from inequality \eqref{ineq:main1} and from inequality \eqref{ineq:cheeger} that
	$$\sigma_2\geq \frac{1}{2}\bigl(-A+\sqrt{A^2+h_\Sigma^2}\bigr).$$
	The interest of this lower bound is that it depends only on the geometry of $M$ on and near its boundary $\Sigma$. Of course this is useful only when the boundary $\Sigma$ is connected, in which case $h_\Sigma>0$.
	It is easy to construct examples of manifolds where $h_Mh_M'$ is arbitrarily small while 
	$\frac{1}{2}\bigl(-A+\sqrt{A^2+h_\Sigma^2}\bigr)$ is bounded away from zero, for instance by attaching a Cheeger dumbbell with thin neck to the interior of $M$ using a thin cylinder.

One could also obtain similar lower bounds for $\sigma_k$ from the  higher order Cheeger inequality for $\lambda_k$ \cite{miclo,funano}, which should be compared with the higher order Cheeger type inequality proved in \cite{HM}. The lower bounds for $\sigma_k$ given in \cite{jammes,HM} depends on the global geometry of the manifold and not only the geometry of the manifold near the boundary.
\end{rem}

\subsection{Signed curvature and convexity}\label{section:introsigned}
In \cite{Xi}, Xiong  extended the results of~\cite{PS16} to domains $\Omega$ with convex  boundary in a complete Riemannian manifold $X$ with either nonpositive ($K_X\leq 0$) or strictly positive ($K_X>0$) sectional curvature, with some hypotheses on the shape operator.
This lead to explicit values for the constants $\constA$ and $\constB$. This work was enlightening for us, and lead to Theorem~\ref{thm:signedcurvature} and Corollary \ref{Xiong} below.

Let us now discuss various geometric settings where the constants $\constA$ and $\constB$ can be improved.
\begin{theorem}\label{thm:signedcurvature}
Let $M$ be a smooth compact manifold of dimension $n+1$ with nonempty boundary $\Sigma$. Let $\hb=\roll(M)$. Let $\lambda,\kappa_+>0$.
\begin{enumerate}
\item Suppose that $\Sigma$ is totally geodesic.
If $|K|\leq\lambda^2$, then \eqref{ineq:main1} and \eqref{ineq:thm:mainstek} hold with
$$\constA=n\max\{\bar h^{-1},\frac{2\lambda}{\pi}\}+\lambda\quad \mbox{and}\quad\constB =\frac{1}{2}\left(\frac{1}{\hb}+{n\lambda}\right).$$
\smallskip

\noindent Knowing the sign of the sectional curvature $K$ leads to slightly more precise bounds.
\begin{itemize}
\item[a)] If $-\lambda^2<K<0$, then \eqref{ineq:main1} and \eqref{ineq:thm:mainstek} hold with
$$\constA=\frac{1}{\bar h}+\lambda\quad\mbox{and}\quad \constB=\frac{1}{2}\left(\frac{1}{\bar h}+n\lambda\right).$$
\medskip
\item[b)] $0<K<\lambda^2$, then \eqref{ineq:main1} and \eqref{ineq:thm:mainstek} hold with
$$\constA=n\max\{\bar h^{-1},\frac{2\lambda}{\pi}\}\quad\text{and}\quad \constB=\frac{1}{2\hb}.$$
\end{itemize}
\item Suppose that $\Sigma$ is a minimal hypersurface. 
If $\Ric\ge\lambda^2n$, then \eqref{ineq:thm:mainstek} holds with
	$\constB$ given in part (1).
	\medskip
	
\item If $-\lambda^2\leq K\leq 0$ and each principal curvature satisfies $\lambda<\kappa_i<\kappa_+$, then \eqref{ineq:main1} and \eqref{ineq:thm:mainstek} hold with
$$\constA=n\max\{\kappa_+,\frac{1}{\hb}\}\qquad\mbox{and}\qquad\constB=\frac{1}{2\hb}.$$
\item 
 If $0<K<\lambda^2$ and each principal curvature satisfies $0<\kappa_i<\kappa_+$, then \eqref{ineq:main1} and \eqref{ineq:thm:mainstek} hold with
$$\constA=n\max\left\{\frac{1}{\hb},\sqrt{\lambda^2+\kappa_+^2}\right\}\quad\text{and}\quad \constB=\frac{1}{2\hb}.$$
\end{enumerate}
\end{theorem}
We recover the results in \cite{Xi} as a consequence of Theorem~\ref{thm:signedcurvature}.
\begin{cor}{\cite[Theorem 1]{Xi}}\label{Xiong}
Let $M$ be a domain in a complete Riemannian manifold with boundary $\Sigma$. Let $\lambda,\kappa_+>0$. Then 
\begin{enumerate}
\item If $-\lambda^2\leq K\leq 0$ and each principal curvature of $\Sigma$ satisfies $\lambda<\kappa_i<\kappa_+$, then \eqref{ineq:main1} and \eqref{ineq:thm:mainstek} hold with
$$\constA=n\kappa_+\qquad\mbox{and}\qquad{\constB=\frac{1}{2}\kappa_+}.$$
\item If $0<K<\lambda^2$ and each principal curvature of $\Sigma$ satisfies $0<\kappa_i<\kappa_+$, then \eqref{ineq:main1} and \eqref{ineq:thm:mainstek} hold with
$$\constA=n\sqrt{\lambda^2+\kappa_+^2}\quad\text{and}\quad\constB= \frac{1}{2}\sqrt{\lambda^2+\kappa_+^2}.$$
\end{enumerate}
\end{cor}
We refer the reader to Corollary \ref{gXiong} for another situation where $\constA$ and $\constB$ have a simpler expression than those in Theorem~\ref{thm:signedcurvature}.

Let us conclude with the situation where $K\equiv 0$, which is motivated by the Euclidean case from \cite{PS16}.
\begin{theorem}\label{thm:flatspace}
Let $M$ be a smooth compact manifold with nonempty boundary $\Sigma$. Suppose that $M$ is flat ($K\equiv 0$).
\begin{itemize}
\item
 If $\kappa_-\leq\kappa\leq\kappa_+$, then \eqref{ineq:main1} and \eqref{ineq:thm:mainstek} hold with $$\constA=n\max\{\hb^{-1},\kappa_+\}+|\kappa_-|\quad \mbox{and} \quad\constB=\frac{1}{2}(\bar h^{-1}+n\vert \kappa_-\vert).$$
\item If $0\leq\kappa_i<\kappa_+$, then \eqref{ineq:main1} and \eqref{ineq:thm:mainstek} hold with $$\constA=n\max\{\hb^{-1},\kappa_+\}\quad \mbox{and} \quad\constB=\frac{1}{2\hb}.$$
\item If $\kappa_-<\kappa_i\leq 0$, then \eqref{ineq:main1} and \eqref{ineq:thm:mainstek} hold with 
$$
\constA= \bar h^{-1}+\vert \kappa_-|\quad\mbox{and}\quad\constB=\frac{1}{2}(\bar h^{-1}+n\vert \kappa_-\vert).
$$
\end{itemize}
\end{theorem}
The proof of Theorem \ref{thm:flatspace} is presented in Section \ref{section:flatmanifolds}.

One could also consider manifolds whose boundary admits a neighbourhood which is isometric to the Riemannian product $[0,L)\times\Sigma$ for some $L>0$, in which case one can take $2\constB=\constA=1/L$. This is discussed in Section (\ref{cylindricalcase}).

\subsection*{Plan of the paper}
The proof of our main result (Theorem \ref{mainintro}) is based on the comparison geometry of principal curvatures of hypersurfaces that are parallel to the boundary. This is presented in Section \ref{section:prelim} following a review of relevant Jacobi fields and Riccati equations. The Pohozaev identity is used in Section \ref{section:pohozaev} to relate the Dirichlet energy of an harmonic function to the $L^2$-norm of its normal derivative. The proof of Theorem \ref{mainintro} is presented in Section \ref{section:mainresults}. In Section \ref{section:signedcurvature} we specialize to various geometrically rigid settings and prove Theorem \ref{thm:signedcurvature}. Here, we give precise values of the constants $\constA,\constB$ occurring in the estimates. In order to do this, we need some $1$-dimensional calculations for specific Riccati equations, which are treated in an appendix at the end of the paper.
In Section \ref{section:examples}, we give various examples to illustrate the necessity of the geometric hypothesis occurring in Theorem \ref{mainintro} and Theorem \ref{thm:signedcurvature}.

\section{Preliminaries from Riemannian geometry}\label{section:prelim}

Let $M$ be a smooth compact Riemannian manifold of dimension $n+1$ with boundary $\Sigma$.
The distance function $f:M\rightarrow\R$ to the boundary $\Sigma$ is given by
$$
f(x)=\dist(x,\Sigma).
$$
Any $s\geq 0$ that is small enough is a regular value of $f$, so that the level sets $\Sigma_s:=f^{-1}(s)$ are submanifolds of $M$, which are called \textit{parallel hypersurfaces}. They are the boundary of $\Omega_s:=\{x\in M\,:\,f(x)\geq s\}$. For $x\in\Sigma_s$, the gradient of the distance function $\grad f(x)$, is the inward normal vector to $\Sigma_s=\partial\Omega_s$:
$$\nu(x):=\grad f(x).$$
In particular, for $x\in\Sigma$, $\nu(x)=\grad f(x)=-\n(x)$. The  distance function $f$ satisfies $|\grad f|=1$, whence the  integral curves of $\grad f$ are geodesics. That is,
$$\nabla_{\grad f}\grad{f}=0.$$

\subsection{Cut locus and rolling radius}
Given $p\in\Sigma$, the exponential map defines a normal geodesic curve $\gamma_p:\R_+\longrightarrow M$, 
$$\gamma_p(s)=E_s(p)=\exp_p(s\nu).$$
The \emph{cut point} $\cut_{\Sigma}(p)$ of $p\in \Sigma$ is the point
$E_{s_0}(p)$, where $s_0>0$ is the first time that $E_s(p)$ stops
minimizing the distance to $\Sigma$. The \emph{cut locus} of $\Sigma$ is
\[\C_{\Sigma} = \{\cut_{\Sigma}(p) : p\in\Sigma\}.\]
The distance between $\Sigma$ and $\C_{\Sigma}$ is called the
\textit{rolling radius}\footnote{It is called the rolling radius
  because any open ball of radius $\le$ $roll(M)$ can roll along
  $\Sigma$ while always remaining a subset of $M$.} of $M$:
\[\roll(M)=\dist(\Sigma,\C_{\Sigma}).\]
Given $h\in (0,\roll(M)]$, define the tubular neighbourhood
$$M_h=\{x\in M: f(x)<h\}.$$
For each $x\in M_h$, there is exactly one nearest point to $x$ in
$\Sigma$ and the exponential map $E_s(p)=\exp_p(s\nu(p))$ defines
a diffeomorphism  between $[0,h)\times\Sigma$ and the tubular
neighbourhood $M_h$.
Moreover, for each $s\in [0,\roll(M))$, the map $E_s:\Sigma\rightarrow\Sigma_s$ is a diffeomorphism.

\subsection{The principal curvatures of parallel hypersurfaces}
For each $s\in[0,\roll (M))$, the \textit{shape operator} associated to the parallel hypersurface $\Sigma_s$ is the endomorphism $S_s:T\Sigma_s\to T\Sigma_s$ defined by 
$$S_s(X)=\nabla_X\grad f.$$
For each tangent vectors $X,Y\in T\Sigma_s$, the following holds:
\begin{equation*} 
 \la S_s(X),Y\ra=\la\nabla_X\grad f,Y\ra=\nabla^2f(X,Y)=\la X,S_s(Y)\ra,
\end{equation*}
where $\la X,Y\ra$ is the Riemannian product and $\nabla^2f$ is the Hessian of $f$. In other words, the Hessian $\nabla^2f$ restricted to $\Sigma_s$ is its second fundamental form. The \emph{principal curvatures}  $\kappa_i(x)$, $i=1,\ldots,n$, of $\Sigma_s$ at point $x$ with respect to outward normal vector $\n(x):=-\grad f(x)$ are the eigenvalues of 
$$-S_s:T_x\Sigma_s\rightarrow T_x\Sigma_s.$$
In other words, the eigenvalues of the Hessian $\nabla^2 f$ at $x\in\Sigma_s$, are 
$$0, -\kappa_1(x),\ldots,-\kappa_n(x).$$
The choice of sign is such that the principal curvatures of the boundary of Euclidean balls are positive.

Our main geometric tools is a comparison estimate for the principal curvatures $\kappa_i$ of the parallel hypersurfaces $\Sigma_s$. This is adapted from the main theorem in \cite{EH90} and from \cite[Proposition 2.3]{Esc87}.

\begin{theorem}[Principal curvature comparison theorem I]\label{thm:shapecomparisonA}
Let $M$ be a smooth compact manifold with nonempty boundary $\Sigma$.
Suppose the sectional curvature satisfies $\alpha\leq K$ and that the principal curvatures of the boundary satisfies 
$\kappa_-\leq\kappa_i$.
Let $a:[0,m(\alpha,\kappa_-))\rightarrow\R$ be the solution of
\begin{equation}
  a'+a^2+\alpha=0, \quad a(0)=-\kappa_-,
\end{equation}
with maximal existence time $m(\alpha,\kappa_-)$. Then $m(\alpha,\kappa_-)\geq\roll(M)$ and for each
$\delta\in(0,\roll(M))$ the shape operator $S_\delta$ associated to the parallel hypersurface $\Sigma_{\delta}$ satisfies
$$S_\delta\leq a(\delta)I.$$
That is, each principal curvature of $\Sigma_{\delta}$ satisfies
$\kappa_i(\delta)\geq -a(\delta)$.
\end{theorem}
\begin{rem}\label{moved}
Theorem \ref{thm:shapecomparisonA} gives an immediate upper bound for $\roll(M)$:
\[\roll(M)\le m(\alpha,\kappa_-).\]
In Lemma \ref{lemma:riccationed} below, an explicit formula for $m(\alpha,\kappa_-)$ is given. It could be compared with the result by Donnelly  and Lee in {\cite[Theorem 3.1]{DL91}} on estimates for $\roll(M)$  on strictly convex domains, i.e. all eigenvalues of $-S$ are positive.
\end{rem}

{We also need a lower bound for the shape operators $S_\delta$ of parallel hypersurfaces in a neighbourhood of the boundary $\Sigma$. This requires further constraints on the neighbourhood.}

\begin{theorem}[Principal curvature comparison theorem II]\label{thm:shapecomparisonB}
Let $M$ be a smooth compact manifold with nonempty boundary $\Sigma$.
Suppose the sectional curvature satisfies $K\leq\beta$
 and that the principal curvatures of the boundary satisfies 
$\kappa_i\leq\kappa_+$.
Let $\beta_+:=\max\{0,\beta\}$, and
let $b:[0,m(\beta_+,\kappa_+))\rightarrow\R$ be the solution of
\begin{equation}\label{eq:defbeta+}
  b'+b^2+\beta_+=0, \quad b(0)=-\kappa_+,
\end{equation}
with maximal existence time $m(\beta_+,\kappa_+)$. Let 
\begin{gather}\label{def:htd}
  \htd:=\min\{m(\beta_+,\kappa_+),\roll(M)\}.
\end{gather}
Then for each
$\delta\in(0,\htd)$ the shape operator $S_\delta$ associated to the parallel hypersurface $\Sigma_{\delta}$ satisfies
$$S_\delta\geq b(\delta)I.$$
That is, each principal curvatures of $\Sigma_{\delta}$ satisfies $\kappa_i\leq-b(\delta)$.
\end{theorem}
\begin{rem}
Theorem \ref{thm:shapecomparisonB} also holds with $\beta_+$ replaced by $\beta$, so that $\htd$ can be replaced by
	$\min\{m(\beta,\kappa_+),\roll(M)\}.$
	 However, in our applications of Theorem \ref{thm:shapecomparisonB}, it is required that we use comparison with a space of non-negative sectional curvature. See
	Lemma~\ref{lemma:eigenvaluehessianeta} and Lemma~\ref{lemma:rhonsmallerone}.
\end{rem}
\begin{rem}
 Theorems \ref{thm:shapecomparisonA} and \ref{thm:shapecomparisonB} compare the evolution of the maximal and minimal principal curvatures of parallel hypersurfaces to those of umbilical hypersurfaces in model spaces of constant curvature $K\equiv\alpha,\beta_+$. Indeed, the functions $-a$ represent the principal curvatures of parallel umbilical hypersurfaces in a space form of constant curvature $K\equiv\alpha$, and similarly for the functions $-b$, for constant curvature $K\equiv\beta_+$.
\end{rem}

In our proofs, we will use mostly the following corollary in combination with explicit formulas for $a(s)$ and $b(s)$.
\begin{corollary}\label{coro:curvaturebounds}
Let $M$ be a smooth compact manifold with nonempty boundary $\Sigma$.
Suppose the sectional curvature satisfies $\alpha\leq K\leq\beta$ and that each principal curvatures along the boundary satisfies $\kappa_-\leq\kappa_i\leq\kappa_+.$
 Then for each $\delta\in(0,\htd)$ the principal curvatures of the parallel hypersurface $\Sigma_{\delta}$ satisfies
$$-a(\delta)\le\kappa_i\le-b(\delta).$$
Moreover, the left-hand side inequality holds for each $\delta\in(0,\bar h)$.
\end{corollary}
%

Theorem \ref{thm:shapecomparisonA} and Theorem \ref{thm:shapecomparisonB} are sufficient to prove a version of Theorem \ref{mainintro}. Nevertheless, the following mean curvature comparison theorem leads to better bounds on the constant $\constB$.
\begin{theorem}[Mean curvature comparison theorem]\label{thm:shapecomparisonC} Let $M$ be a smooth compact manifold with nonempty boundary $\Sigma$.
Suppose the Ricci curvature satisfies $ \Ric\ge n\alpha$ and that the mean curvature $H=-\frac{\tr S_0}{n}$ of the boundary satisfies 
$H\ge \kappa_-$.
Let $\mu:[0,m(\alpha,\kappa_-))\rightarrow\R$ be the solution of
\begin{equation}
  \mu'+\mu^2+\alpha=0, \quad \mu(0)=-\kappa_-,
\end{equation}
with maximal existence time $m(\alpha,\kappa_-)$. Then $m(\alpha,\kappa_-)\geq\hb$ and for each
$\delta\in(0,\hb)$ the mean curvature of the parallel hypersurface $\Sigma_{\delta}$ satisfies
$$H(\delta)>-\mu(\delta).$$
\end{theorem}

%

\subsection{Jacobi fields and Riccati equations}
For the convenience of the reader, we explain briefly in this section how Theorem \ref{thm:shapecomparisonA} and Theorem \ref{thm:shapecomparisonB} are direct consequences of \cite{Esc87} and \cite{EH90}.

Given $v\in T_{p_0}\Sigma$, let $p:(-\epsilon,\epsilon)\rightarrow\Sigma$ be a smooth curve with $p(0)=p_0$ and $p'(0)=v$. Then 
$$J(s):=\frac{d}{dt}E_s(p(t))|_{t=0}\in T\Sigma_s$$
defines a Jacobi field along the normal geodesic $\gamma(s):=E_s(p_0)$. 
This means that $J(s)$ satisfies the Jacobi equation
\begin{gather}\label{equation:jacobi}
J''(s)+R_\nu(J(s))=0,
\end{gather}
where $R_\nu(J(s)):=R(J(s),\nu)\nu$ is the curvature tensor in the normal direction $\nu=\grad f$ to $\Sigma_s$. The field $J$ satisfies the initial conditions
$$J(0)=v, \qquad {J}'(0)=S(v).$$
The shape operator is a tensor in $T\Sigma\otimes T^\star\Sigma$, which satisfies
\begin{gather}\label{evolutionshapejacobi}
S_s(J(s))=J'(s).
\end{gather}
Its covariant derivative is given by
\begin{gather}\label{covariantDXS}
(D_XS)V=D_X(SV)-SD_XV.
\end{gather}
Differentiating \eqref{evolutionshapejacobi} and using \eqref{covariantDXS} and the Jacobi equation \eqref{equation:jacobi} leads to the Riccati equation 
\begin{equation}
\label{riequ1}{S}_s'(J(s))+S_s^2(J(s))+R_\nu(J(s))=0.
\end{equation}
  Given $p_0\in\Sigma$, consider the integral geodesic curve $\gamma:[0,\cut_{\Sigma}(p_0)]\longrightarrow M$ of $\grad f$ starting at $\gamma(0)=p_0$. Identifying vectors in $T_{\gamma(s)}M$ with $T_{p_0}M$ via parallel transport along $\gamma$, we  can consider $S_s$ and $R_\nu$ as endomorphisms on a single vector space $T_{p_0}\Sigma$. They satisfy the following matrix-valued Riccati equation in $\mbox{End}(T_{p_0}\Sigma)$:
\begin{equation}\label{mriccati}
S'(s)+S^2(s)+R(s)=0.
\end{equation}
See \cite[Chapter 3]{Gray} for an enlightening discussion of Riccati equations for shape operators.

Let $E$ be a finite-dimensional real vector space with an euclidean inner product $\la\cdot,\cdot\ra$. Let $S(E)$ be the space of self-adjoint endomorphisms over $E$. For $A,B\in S(E)$, we say $A\le B$ if $B-A$ is positive semi definite.
Then, we have
\begin{theorem}[Riccati comparison theorem \cite{EH90}]\label{thmriccati}
  \label{riccati}
  Let $R_1,R_2 : \R\to S(E)$ be smooth curves with $R_1 \ge R_2$. Let  $S_i : [s_0, s_i)\to S(E)$ be a solution of
    \begin{equation}\label{riequ}S_i' +S_i^2 +R_i = 0,\quad i=1,2 \end{equation}
    with maximal  existence time $s_i \in(s_0,\infty]$. Assume that $S_1(s_0) \le S_2(s_0)$. Then $s_1 \le s_2$ and $S_1(s) \le S_2(s)$ on $(s_0, s_1)$.
\end{theorem}
 
\begin{proof} [\bf Proof of Theorem \ref{thm:shapecomparisonA} and Theorem \ref{thm:shapecomparisonB}]
Let $\mathbb M^{n+1}$ be a space form of sectional curvature $K$. Then $R=KI_n$ in \eqref{mriccati}, where $I_n$ is the identity matrix. Moreover, if we assume that $\mathbb M$ is simply connected and   $\Sigma$ is an umbilical hypersurface i.e. $S(0)=-\kappa I_n$, then the Riccati equation \eqref{mriccati} reduces to the one dimensional problem
\begin{equation}\label{1driccati}
    \begin{cases}
      y'+y^2+K=0\\
      y(0)=-\kappa.
    \end{cases}
  \end{equation}
Its solutions describe the shape operator of a family of parallel umbilical hypersurfaces in $\mathbb M$. The idea of the proof is to compare the situation of the manifold $M$ with the situation of an umbilical hypersurface on a space form, where it is possible to do exact computations. Theorems \ref{thm:shapecomparisonA} and \ref{thm:shapecomparisonB} say that we can estimate the situation on $M$ by comparison with umbilical hypersurfaces on a space form, and in the sequel (Lemma \ref{lemma:riccationed}) we will produce exact computations in these particular situations.
	
	\medskip
	For the proof of Theorem \ref{thm:shapecomparisonA}, we apply Theorem \ref{thmriccati} with  $R_1={R}$, where  $R$  is given in \eqref{mriccati}, $R_2=\alpha I$, $S_1(0)=S(0)$ and $S_2(0)=-\kappa_-I$. By hypothesis, we have $R_1\ge R_2$ and $S_1(0)\le S_2(0)$, so that we can apply Theorem \ref{thmriccati}. Moreover, $s_1\le s_2$ implies that the principal curvature of the hypersurfaces level degenerate before $s_2$ so that the rolling radius has to be less than $s_2 =m(\alpha, \kappa_-)$.
	
	\smallskip
	For the proof of Theorem \ref{thm:shapecomparisonB}, we apply Theorem \ref{thmriccati} with $R_1=\beta_+ I$, $R_2=R$, $S_1(0)=-\kappa_+I$ and $S_2(0)=S(0)$. By hypothesis, we have $R_1\ge R_2$ and $S_1(0)\le S_2(0)$, so that we again apply Theorem \ref{thmriccati}. However the results of Theorem \ref{thmriccati} are available only when the solution on $M$ is defined: we have to be less than $s_1=m(\beta_+,\kappa_+)$ but also less than $\roll(M)$.
\end{proof}
%
%
\begin{proof}[{\bf Proof of Theorem \ref{thm:shapecomparisonC}}]
Taking the trace of  the Riccati equation \eqref{mriccati} and using the Schwarz inequality for endomorphisms $n\,\tr S^2\ge (\tr S)^2$ leads to the following Riccati inequality:
$$\frac{\tr(S(s))'}{n}+\left(\frac{\tr(S(s))}{n}\right)^2+\frac{\tr R}{n}\le0.$$
This can be compared with the one dimensional Riccati problem
\begin{equation*} 
    \begin{cases}
      \mu'+\mu^2+\alpha=0\\
      \mu(0)=-\kappa_-
    \end{cases}
  \end{equation*}
  using \cite[Corollary 1.6.2]{Kar89} to get the desired inequality. See \cite[Pages 181-182]{Kar89} and \cite{Esc87} for details.
\end{proof}

%
\subsection{The Riccati equation for umbilic hypersurfaces in spaceforms}
In order to apply Corollary \ref{coro:curvaturebounds} it will be useful to know the solutions of the one-dimensional Riccati equation explicitly. The following Lemma is adapted from \cite{Esc87}.

\begin{lemma}\label{lemma:riccationed}
  Let $K,\kappa\in\R$. Consider the Riccati initial
  value problem (\ref{1driccati}):
\begin{eqnarray} 
    \begin{cases}
      y'+y^2+K=0\\
      y(0)=-\kappa.
    \end{cases}
  \end{eqnarray}
  Let $\lambda=\sqrt{|K|}\geq 0$.
  The solutions of \eqref{1driccati}, together with maximal existence time $m(K,\kappa)$, are as
  follows:

\smallskip

\textbf{a}) In Euclidean spaces $K=0$, the umbilical hypersurfaces are given by spheres and hyperplanes.

  \begin{tabular}{LCL}
  \mbox{\bf Description}&\mbox{\bf Solution }y(s)& m(K,\kappa)\\
  \hline
    \mbox{Contracting spheres } \kappa> 0&\frac{1}{s-\frac{1}{\kappa}}&1/\kappa\\[.2cm]
    \mbox{Expanding spheres }\kappa< 0&\frac{1}{s-\frac{1}{\kappa}}&\infty\\[.2cm]
    \mbox{Parallel hyperplanes } \kappa=0&0&\infty\\[.2cm]
 \end{tabular}

\textbf{b}) In the sphere $\mathbb{S}_{K}$ of constant curvature $K=\lambda^2$ with $\lambda>0$, the umbilical hypersurfaces are geodesic spheres. 

  \begin{tabular}{CLL}
    \multicolumn{2}{c}{$y(s)=-\lambda\tan\left(\lambda   s+\arctan(\frac{\kappa}{\lambda})\right)$}&\mbox{ with }m(K,\kappa)=\frac{1}{\lambda}\left(\frac{\pi}{2}-\arctan(\frac{\kappa}{\lambda})\right){=\frac{1}{\lambda}\arcot(\frac{\kappa}{\lambda})}\\[.2cm]
 \end{tabular}
   
\textbf{c}) In the hyperbolic space $\mathbb{H}_{K}$ of constant curvature $K=-\lambda^2<0$ with $\lambda>0$, the umbilical hypersurfaces are the following.

\begin{tabular}{LLL}
   \mbox{\bf Description}&\mbox{\bf Solution }y(s)& \hspace{-1cm}\mbox{\bf Maximal existence }m(K,\kappa)\\
\hline
\mbox{Hyperbolic subspace } \kappa=0&\lambda\tanh\left(\lambda s\right)&\infty\\[.2cm]
  \mbox{Hypercycle } |\kappa|<\lambda&\lambda\tanh\left(\lambda s-\arctanh(\frac{\kappa}{\lambda})\right)&\infty\\[.2cm]
  \mbox{Expanding spheres }\kappa<-\lambda&\lambda\coth\left(\lambda s-\arcoth(\frac{\kappa}{\lambda})\right)&\infty\\[.2cm]
\mbox{Contracting spheres }\kappa>\lambda&\lambda\coth\left(\lambda s-\arcoth(\frac{\kappa}{\lambda})\right)&\frac{1}{\lambda}\arcoth(\frac{\kappa}{\lambda})\\[.2cm]
  \mbox{Horospheres } |\kappa|=\lambda&\kappa&\infty
  \end{tabular}
\end{lemma}
The proof of Lemma
\ref{lemma:riccationed} is presented in an appendix.
Let us conclude this paragraph with a technical lemma which will allow the control near the cut locus, in the situation where the principal curvatures of the level hypersurfaces could degenerate.
\begin{lemma}\label{lemma:rhonsmallerone}
  Let $K\in\R_{\ge 0}$ and $\kappa\in\R$. Let $y:I\subset\R\to \R$ be a solution of the Riccati equation

\begin{equation}
  y'+y^2+K=0, \quad y(0)=-\kappa,
 \end{equation}
    with maximal interval
  $I$. Then  $-(h-s)y(s)\le1$ for any $0\le s<h\in I$.
\end{lemma}
The proof of Lemma \ref{lemma:rhonsmallerone} will be presented in the appendix.

\section{Pohozaev identity and its application}
\label{section:pohozaev}

Let $u\in C^\infty(M)$ be an harmonic function. The main goal of this section is to obtain a quantitative comparison inequality relating the norms $\|\grad_\Sigma u\|_{L^2(\Sigma)}$ and $\|\frac{\partial u}{\partial\n}\|_{L^2(\Sigma)}$. Here $\grad_\Sigma u$ denotes the tangential gradient of a function $u\in H^1(\Sigma)$ which is the gradient of $u$ on $\Sigma$.  To achieve this goal, we need the generalized Pohozaev identity  for harmonic functions on $M$. 

\begin{lemma}[Generalized Pohozaev identity]\label{poho}
Let $F:M\rightarrow TM$ be a Lipschitz vector field. Let $u\in C^\infty(M)$ with $\Delta u=0$ in $M$. Then
\begin{multline}\label{pohozaev}
0=\int_{\Sigma}\frac{\partial u}{\partial \n}\langle F,\grad u\rangle\, dV_\Sigma-\frac{1}{2}\int_{\Sigma}|\grad u|^2\la F,\n \ra\, dV_\Sigma\\
+\frac{1}{2}\int_{M}|\grad u|^2 {\rm div}F\, dV_M-\int_{M}\la \nabla_{\grad u}F,\grad u\ra\, dV_M,
\end{multline}
where $\nabla F$ denotes the covariant derivative  of $F$, $\n$ is the unit outward normal vector field along $\Sigma$, and $dV_M$ and $dV_\Sigma$ are Riemannian volume elements of $M$ and $\Sigma$ respectively.
\end{lemma}
The proof of Lemma \ref{poho} was provided first in \cite[Lemma 3.1]{PS16} for Euclidean domains. For simply-connected domains in a complete Riemannian manifold a proof was provided in \cite[Lemma 9]{Xi}. See also \cite[Section 3]{HS}, where a more general version of \eqref{pohozaev}, so-called the Rellich identity is discussed. For the sake of completeness, we also include a proof here, which works for each compact manifold with boundary.
\begin{proof}[Proof of Lemma \ref{poho}]
Working in normal coordinates at a point $p\in M$, we can proceed exactly as in $\R^n$ since only the value of the metric and of its first order derivative at $p$ are involved in the following local computation.
It follows from $\Delta u=0$ that
$$\divergence \left(\la F,\nabla u\ra\nabla u\right)=\la\nabla \la F,\nabla u\ra,\nabla u\ra=
\la\nabla_{\nabla u}F,\nabla u\ra+\nabla^2 u(F,\nabla u).$$
Moreover,
$$\divergence |\nabla u|^2F=2\nabla^2u(F,\nabla u)+|\nabla u|^2\divergence F.$$
Therefore
$$\divergence \left(\la F,\nabla u\ra\nabla u-\frac{1}{2}|\nabla u|^2F\right)=
\la\nabla_{\nabla u}F,\nabla u\ra-\frac{1}{2}|\nabla u|^2\divergence F.$$
Integrating this identity on $M$ and using the divergence theorem completes the proof.
\end{proof}

\subsection{Analytic estimate using distance function}
Recall that $f:M\rightarrow\R$ is defined by $f(x)={\rm dist}(x,\Sigma$).
Given $h\in(0,\roll(M))$ let $\tilde f:M_h\rightarrow\R$ be the distance function to $\Sigma_h$.
That is
 $$
\tilde f(x):=h-f(x)=
{\rm dist}(x,\Sigma_h).
$$
\begin{lemma}\label{lemma:analyticestimatedistance}
 Let $u\in C^\infty(M)$ be such that $\Delta u=0$ and let
  $\eta=\frac{1}{2}\tilde{f}^2$. Then
  \begin{gather}\label{mformul}
\int_{\Sigma}|\nabla_\Sigma u|^2 -\left(\frac{\partial u}{\partial\n}\right)^2 dV_\Sigma=
\frac{1}{h}\int_{M_h}\left(\Delta\eta |\grad u|^2  -2\nabla^2\eta(\grad u,\grad u)\right) dV_M.
\end{gather}
\end{lemma}
\begin{proof}
Define the vector field $F:M\rightarrow TM$ by
\begin{equation}\label{field}
F(x):=\begin{cases}
\grad{\eta}(x) & {\rm if\ }x\in M_h,\\
0& {\rm if\ }x\in M\setminus M_h.
\end{cases}
\end{equation}
Because $\grad \eta=\tf\grad \tf$ on $M_h$, the vector field $F$ is
Lipschitz on $M$. Note that for $u\in C^\infty(M)$ and any $p\in\Sigma$,
$$|\grad u(p)|^2=|\grad_\Sigma u(p)|^2+\left(\frac{\partial u}{\partial\n}(p)\right)^2.$$
 \smallskip
\noindent
It follows from the Pohozaev identity \eqref{pohozaev} that
\begin{align*}
  2\int_{\Sigma}\frac{\partial u}{\partial\n}\langle \grad \eta,\grad
  u\rangle\, dV_\Sigma&-\int_{\Sigma}|\grad_\Sigma u|^2 \la \grad
  \eta,\n \ra\, dV_\Sigma\\
  -\int_{\Sigma} &\left(\frac{\partial u}{\partial\n}\right)^2\la \grad \eta,\n \ra\, dV_\Sigma+
  \int_{M}|\grad u|^2 \Delta \eta\, dV_M\\
  &-2\int_{M}\la \nabla_{\grad u}\grad \eta,\grad u\ra\, dV_M=0.
\end{align*}
Since $\grad\eta|_{\Sigma}=h\n$, we get
\begin{align*}
2h\int_{\Sigma}\left(\frac{\partial u}{\partial\n} \right)^2 dV_\Sigma&- h\int_{\Sigma}|\grad_\Sigma u|^2   dV_\Sigma
-h\int_{\Sigma} \left(\frac{\partial u}{\partial\n}\right)^2  dV_\Sigma\\
&+
  \int_{M}|\grad u|^2 \Delta \eta\,dV_M-2\int_{M}\la \nabla_{\grad u}\grad \eta,\grad u\ra\,dV_M=0.
\end{align*}
\end{proof}

\subsection{Geometric control on a tubular neighbourhood}
Let $M$ be a smooth compact Riemannian manifold  with nonempty boundary $\Sigma$ and rolling radius $\hb=\roll(M)$. 
Suppose the sectional curvature satisfies $\alpha\leq K\leq\beta$ on $M_{\hb}$ and suppose the principal curvatures of $\Sigma$ satisfy $\kappa_-\leq\kappa_i\leq\kappa_+$.

The link between the function $\eta=\frac{1}{2}\tf^2$ and the local geometry of parallel
hypersurfaces is established in the following lemma.
\begin{lemma}\label{lemma:eigenvaluehessianeta}
Let $h\in (0,\htd)$ and let $\delta\in (0,h)$.
  Then the eigenvalues of the Hessian $\nabla^2\eta$ at $y\in\Sigma_{\delta}$ are 
  \begin{gather*}
    \rho_1=(h-\delta)\kappa_1(y)\leq\ldots\leq\rho_n=(h-\delta)\kappa_n(y)\leq\rho_{n+1}=1.
  \end{gather*}
\end{lemma}
\begin{proof}
The Hessian of $\eta$ is 
 \[\nabla^2\eta=\nabla \tf\otimes\nabla \tf+\tf\nabla^2\tf.\]
 The eigenvalues of $\nabla^2\tf(y)$ are $0,
 \kappa_1(y),\ldots,\kappa_n(y)$.
 Moreover, for each $V,W\in T_{y}\Sigma_{\delta}$,
 $$\nabla \tf\otimes\nabla \tf(V,W)=0.$$ 
It follows from
$$\Sigma_{\delta}=\tf^{-1}(h-\delta)$$
that for each $i=1,\cdots,n$ the number $\rho_i:=(h-\delta)\kappa_i(y)$ is an eigenvalue of $\nabla^2f(y)$.
Finally, it follows from Corollary \ref{coro:curvaturebounds} that $(h-\delta)\kappa_i(y)\leq -(h-\delta)b(\delta)$, where $b(\delta)$ is defined in \eqref{eq:defbeta+} using $\beta_+\geq 0$. The proof is now completed using Lemma \ref{lemma:rhonsmallerone}.
  \end{proof}

\begin{lemma}\label{lemma:firstgeomcontrol}
  For each smooth function $v\in C^\infty(M)$, the following pointwise
  estimate holds on $M_\hb$,
  \begin{gather}\label{ineq:locquadratic}
  -(1-\sum_{i=1}^n \rho_i)|\grad v|^2 \le \Delta{\eta}|\grad v|^2-2\nabla^2{\eta}(\grad v,\grad v)\leq (1+\sum_{i=2}^{n}\rho_i-\rho_1)|\grad v|^2.
  \end{gather}
\end{lemma}
\begin{proof}
  Let $A=\mbox{diag}(\rho_i)$ be the diagonal matrix
  representing the Hessian of $\eta$ at $x$ in an orthonormal frame. Let
  $w=\grad v(x)$. Observe that
  $$\rho_1|w|^2\leq Aw\cdot w\leq\rho_{n+1}|w|^2$$
  It follows that
  $$(\rho_1+\cdots+\rho_n-\rho_{n+1})|w|^2\leq|w|^2\tr(A)-2Aw\cdot w\leq(-\rho_1+\rho_2+\cdots+\rho_{n+1})|w|^2,$$
  with $\rho_{n+1}=1$.
\end{proof}


\begin{lemma}\label{ahestimate}
Let $M\in \mathcal{M}(n,\alpha,\beta,\kappa_+,\kappa_-,\bar h)$.
Then there exist  constants $\cbm=\cbm(n,\alpha,\beta,\kappa_-,\kappa_+)$ and $\scbm=\scbm(n,\alpha,\kappa_-)$ such that every $v\in H^1(M)$ satisfies the following inequality	for each $h\in (0,\htd)$,
  \begin{multline}\label{I}
 -(1+\scbm)\int_{M}|\grad v|^2\,dV_M\le \int_{M_h}|\grad v|^2\Delta{\eta}-2\nabla^2{\eta}(\grad v,\grad v)\,dV_M\\\leq{(1+ \cbm)}\int_{M}|\grad v|^2\,dV_M.
  \end{multline}
  In addition, the left-hand side inequality holds for each $h\in (0,\hb)$.
\end{lemma}
\begin{proof}
Let $h\in(0,\htd)$. Given $\delta\in (0,h)$, let $H(y)$ be the mean curvature of the parallel hypersurface $\Sigma_\delta$ at $y\in\Sigma_\delta$. It follows from Lemma  \ref{lemma:eigenvaluehessianeta} and  Lemma \ref{lemma:firstgeomcontrol} that
\begin{multline}
-(1-(h-\delta)nH)|\grad v|^2\le \Delta{\eta}|\grad v|^2-2\nabla^2{\eta}(\grad v,\grad v)
\\\leq
\left(1+(h-\delta)\sum_{i=2}^{n}\kappa_i-(h-\delta)\kappa_1\right)|\grad v|^2.
\end{multline}
Following Corollary \ref{coro:curvaturebounds} and Theorem~\ref{thm:shapecomparisonC},
	let
\begin{gather}\label{constantAh}
\cbm(h)=\max_{0\leq\delta<h}\left\{(h-\delta)\Bigl(a(\delta)-(n-1)b(\delta)\Bigr)\right\},
\end{gather}
and 
\begin{equation}\label{constantBh}
\scbm(h)=n\max_{0\leq\delta<h}(h-\delta)\mu(\delta)\blue{,}
\end{equation}
so that for each $h\in(0,\htd)$ one has
 \begin{multline*}
-(1+\scbm(h))\int_{M}|\grad v|^2\,dV_M\le \int_{M_h}|\grad v|^2\Delta{\eta}-2\nabla^2{\eta}(\grad v,\grad v)\,dV_M\\\leq{(1+ \cbm(h))}\int_{M}|\grad v|^2\,dV_M.
\end{multline*}
Notice that according to Theorem~\ref{thm:shapecomparisonC}, the left-hand side inequality holds for each $h\in(0,\hb)$. One can therefore take
\begin{gather}\label{constant:cbm}
\cbm\ge\max_{0\leq h<\htd}\cbm(h),
\end{gather}
and 
\begin{equation}\label{constant:scbm}
\scbm\ge\max_{0\leq h<\hb}\scbm(h).
\end{equation}
It remains to prove that the right-hand sides of \eqref{constant:cbm} and \eqref{constant:scbm} are finite. To achieve this,
we provide upper bounds for $\bar A(h)$ and $\bar B(h)$.

For each $\delta\in (0,h)$, the number $-a(\delta)$ is the principal curvature of an umbilical hypersurface of principal curvature $\kappa_-$ in the space form of constant curvature $\alpha$. Observe that for each $\epsilon>0$,
$$\alpha\ge-(\overbrace{|\alpha|+\kappa_-^2+\epsilon}^{\hat{\alpha}}).$$
Using Corollary \ref{coro:curvaturebounds} one more time we see that
$a(\delta)$ is bounded above by the solution $y(\delta)$ of
$$y'+y^2-\hat{\alpha}=0,\qquad y(0)=-\kappa_-.$$
Because $|\kappa_-|\le\sqrt{\hat{\alpha}}$, Lemma \ref{lemma:riccationed} indicates that $y(\delta)$ is the principal curvature of an hypercycle in $\mathbb{H}_{\sqrt{\hat{\alpha}}}$ and it follows that for each $\delta\in(0,h)$,
$$a(\delta)\leq\sqrt{\hat{\alpha}}\tanh(\star)\leq\sqrt{|\alpha|+\kappa_-^2}.$$
\alx{Here $\star$ stands for the argument but is not relevant to the estimate.}
This bound does not depend on $0<h<\hb$.
We proceed similarly to bound $-(h-\delta)b(\delta)$ for each $h\in(0,\htd)$. Indeed, the situation is simpler since we have defined $b(\delta)$ using a space of constant curvature $\beta_+\geq 0$. In this case it follows from Lemma \ref{lemma:rhonsmallerone} that $-(h-\delta)b(\delta)\leq 1$.
Substitution in the definition \eqref{constantAh} leads to
$$\cbm(h)\leq (\htd-h)\sqrt{|\alpha|+\kappa_-^2}+(n-1)$$
This implies that \eqref{I} holds for any
\begin{gather}\label{constant:explicitAbar}
\cbm\geq\htd\sqrt{|\alpha|+\kappa_-^2}+(n-1).
\end{gather}
One proceed exactly in the same way to bound $(h-\delta)n\mu(\delta)$ and obtain
$$(h-\delta)n\mu(\delta)\leq(\hb-h)\sqrt{|\alpha|+\kappa_-^2}.$$
This implies that \eqref{I} holds for any
\begin{gather}\label{constant:explicitBbar}
\scbm\geq\hb n\sqrt{|\alpha|+\kappa_-^2}.
\end{gather}
This completes the proof.
\end{proof}
Note that to bound $\bar B(h)$ from above we use Theorem \ref{thm:shapecomparisonC}, where only lower bounds on the Ricci and on the mean curvature are needed. Hence, we can weaken the hypotheses in Lemma \ref{ahestimate}.
\begin{lemma}\label{estimatew/ricci}
Let $M\in \mathcal{M}'(n,\alpha,\kappa_-,\hb)$.
Then there exists  constant  $\scbm=\scbm(n,\alpha,\kappa_-)$ such that every $v\in H^1(M)$ satisfies the following inequality	for each $h\in (0,\hb)$,
  \begin{equation*} -(1+\scbm)\int_{M}|\grad v|^2\,dV_M\le \int_{M_h}|\grad v|^2\Delta{\eta}-2\nabla^2{\eta}(\grad v,\grad v)\,dV_M.
  \end{equation*}
\end{lemma}
\begin{rem}\label{Bvalue0} When the lower bound for the mean curvature of the parallel hypersurfaces is nonnegative or equivalently when $\mu(\delta)$ is nonpositive for every $0\le \delta\le\hb$, then we can simply take $\scbm=0$.
\end{rem}

\begin{theorem}\label{equi1}
The assumptions are the same as in Lemma \ref{ahestimate}.

\noindent
Let $\constA,\constB\in\R$ be such that 
\begin{gather}\label{constant:CM}
\constA\geq \frac{1}{\htd}(1+ \cbm),\quad \text{and}\quad \constB\geq\frac{1}{2\hb}(1+ \scbm).
\end{gather}
Let $v\in C^\infty(M)$ be such that $\Delta v=0$ in $M$ and normalized such that $\int_{\Sigma}v^2 dV_\Sigma=1$. Then the following holds

\begin{gather}\label{tan}
\int_{\Sigma}|\grad_\Sigma v|^2dV_\Sigma\leq \int_{\Sigma}\left(\frac{\partial v}{\partial\n}\right)^2dV_\Sigma+{\constA}\left(\int_{\Sigma}\left(\frac{\partial v}{\partial\n}\right)^2dV_\Sigma\right)^{\frac{1}{2}}.
\end{gather}

\begin{gather}\label{nor}
\left(\int_{\Sigma}\left(\frac{\partial v}{\partial\n}\right)^2dV_\Sigma\right)^{\frac{1}{2}}\leq{\constB}+\sqrt{\constB+\int_{\Sigma}|\grad_\Sigma v|^2\,dV_\Sigma}.
\end{gather}
\end{theorem}
\begin{proof}
This is identical to the proof of Theorem 3.18 in \cite{PS16}, and included for the sake of completeness.
	 
\noindent
Proof of \eqref{tan}.
It follows from Lemma \ref{lemma:analyticestimatedistance} and
Lemma~\ref{ahestimate} that
\begin{align*}
\int_{\Sigma}|\nabla_\Sigma v|^2-\left(\frac{\partial v}{\partial\n}\right)^2\,dV_\Sigma&=
\frac{1}{h}\int_{M_h}|\grad v|^2 \Delta\eta -2\nabla^2\eta(\grad v,\grad v)\,dV_M\\
&\le\constA\int_{M}|\grad v|^2\,dV_M=\constA\int_{\Sigma}v\frac{\partial v}{\partial\n}\,dV_\Sigma.
\end{align*}
It follows from the Cauchy-Schwarz inequality that
$$\int_{\Sigma}|\nabla_\Sigma v|^2-\left(\frac{\partial v}{\partial\n}\right)^2\,dV_\Sigma\le\constA\left(\int_{\Sigma}\left(\frac{\partial v}{\partial\n}\right)^2 dV_\Sigma\right)^{1/2}.$$
This completes the proof of \eqref{tan}.\\

\noindent
Proof of \eqref{nor}. We repeat the calculation above using Lemma~\ref{lemma:analyticestimatedistance} and Lemma~\ref{ahestimate}.
\begin{align*}
\int_{\Sigma}|\nabla_\Sigma v|^2 -\left(\frac{\partial v}{\partial\n}\right)^2 dV_\Sigma
&=\frac{1}{h}\int_{M_h}|\grad v|^2 \Delta\eta-2\nabla^2\eta(\grad v,\grad v)\, dV_M\\
&\ge-2\constB\int_{M}|\grad v|^2\,dV_M=-2\constB\int_{\Sigma}v\frac{\partial v}{\partial\n} \,dV_\Sigma.
\end{align*}
It follows from the Cauchy-Schwarz inequality that
$$\int_{\Sigma}|\nabla_\Sigma v|^2 -\left(\frac{\partial v}{\partial\n}\right)^2 dV_\Sigma\ge-2\constB\left(\int_{\Sigma}\left(\frac{\partial v}{\partial\n}\right)^2 dV_\Sigma\right)^{1/2}.$$
We solve this inequality  in terms of the unknown $\left(\int_{\Sigma}\left(\frac{\partial v}{\partial\n}\right)^2dV_\Sigma\right)^{1/2}$ and get
\[\left(\int_{\Sigma}\left(\frac{\partial v}{\partial\n}\right)^2 dV_\Sigma\right)^{1/2}\le \constB+\sqrt{\constB^2+\int_{\Sigma}|\nabla_\Sigma v|^2\,dV_\Sigma}.\]
\end{proof}
\begin{rem}\label{remw/ricci}
Note that inequality \eqref{nor} is also true under the assumption of Lemma~\ref{estimatew/ricci}.
\end{rem}

\begin{rem}
It follows from the above that for each $f\in C^\infty(\Sigma)$ and $u=\mathcal{H}f\in C^\infty(M)$,
$$\bigl|\|\nabla_\Sigma f\|^2-\|Df\|^2\bigr|\leq \max\{A,2B\}\int_{M_h}|\nabla u|^2\,dV_M.$$
The tangential gradient and the DtN map have comparable norm, the difference being controlled by the Dirichlet energy of the harmonic extension. 
\end{rem}
\begin{rem}
If we keep the norm $\|v\|_\Sigma$ in the above proof of \eqref{nor} we get from the Cauchy-Schwarz inequality that
$$\int_{\Sigma}|\nabla_\Sigma v|^2 -\left(\frac{\partial v}{\partial\n}\right)^2 dV_\Sigma\ge-2\constB\left(\int_{\Sigma}v^2\,dV_\Sigma\right)^{1/2}\left(\int_{\Sigma}\left(\frac{\partial v}{\partial\n}\right)^2 dV_\Sigma\right)^{1/2}.$$
And solving the inequality leads to
\[\|Dv\|_{L^2(\Sigma)}\le \constB\|v\|_{L^2(\Sigma)}+\sqrt{\constB^2\|v\|_{L^2(\Sigma)}^2+\|\nabla_\Sigma v\|^2}\leq(\sqrt{\constB^2+1})\|v\|_{H^1(\Sigma)}.\]
This proves that the DtN map $D:H^1(\Sigma)\rightarrow L^2(\Sigma)$ is a bounded map and that its operator norm is smaller than $\sqrt{\constB^2+1}$. It is well-known that $D$ is bounded in this setting, but our result provide an explicit bound on its norm.
\end{rem}

%
%
\section{Proof of the main result}\label{section:mainresults}
We are now ready for the proof of the main result. The proof follows the same lines of argument as in \cite[Theorem 1.7]{PS16}. For the sake of completeness we give the proof here.

  
\begin{proof}[Proof of Theorem \ref{mainintro}]
Let us first recall the variational characterizations of the eigenvalues of Dirichlet-to-Neumann operator $\sigma_k$ and of the Laplacian $\lambda_k$. For each $k\in \N$,
\begin{equation}\label{stekminimax}
\sigma_k=\inf_{\substack{V\subset H^1(M),\\{\rm dim}V=k}}\sup_{\substack{v\in V,\\\int_{\Sigma}v^2dV_\Sigma=1}}\int_{M}|\grad v|^2dV_M\,;
\end{equation}
\begin{equation}\label{minimax}
\lambda_k=\inf_{\substack{V\subset H^1(\Sigma),\\{\rm dim}V=k}}\sup_{\substack{v\in V,\\\int_{\Sigma}v^2dV_\Sigma=1}}\int_{\Sigma}|\grad_\Sigma v|^2dV_\Sigma.
\end{equation}

\noindent
\textbf{Part a}.
Let $(\psi_i)_{i\in\N}\subset H^1(M)$ be a complete set of eigenfunctions corresponding to the Steklov eigenvalues $\sigma_i(M)$ such that their restrictions to $\Sigma$ form an orthonormal basis of $L^2(\Sigma)$. We use the space spanned by $\{\psi_i\bigl|\bigr._{\Sigma}\}_{i=1}^k$ as a trial space $V$ in \eqref{minimax}. Every $v\in V$ with $\int_{\Sigma} v^2dV_\Sigma=1$ can be written as $v=\sum c_i\psi_i\bigl|\bigr._{\Sigma}$ with $\sum c_i^2=1$. Note that $$\frac{\partial \psi_i}{\partial \n}=\sigma_i\psi_i,\quad\mbox{ on }~\Sigma\quad\mbox{ for }\quad i=1,\ldots, k.$$
Hence, using inequality \eqref{tan} we get 
\begin{align*}
\lambda_k&\le\sup_{\substack{0\ne u\in V\\\int_{\Sigma}u^2dV_\Sigma=1}}\int_{\Sigma}|\grad_\Sigma v|^2dV_\Sigma\\
&\le \sup_{\substack{0\ne u\in V\\\int_{\Sigma}u^2dV_\Sigma=1}} \int_{\Sigma}\left(\frac{\partial v}{\partial\n}\right)^2dV_\Sigma+{\constA}\left(\int_{\Sigma}\left(\frac{\partial v}{\partial\n}\right)^2dV_\Sigma\right)^{\frac{1}{2}}\\
&=\sup_{\substack{(c_1,\cdots,c_k)\in \Sp^{k-1}}}\left(\sum_{i=1}^kc_i^2\sigma_i^2+{\constA}\left(\sum_{i=1}^kc_i^2\sigma_i^2\right)^{\frac{1}{2}}\right)=\sigma_k^2+{\constA}\sigma_k.
\end{align*}
This proves inequality \eqref{ineq:main1}. 
%
%

\noindent
\textbf{Part b}.
We prove inequality \eqref{ineq:thm:mainstek} in an analogous way.  Let $\{\varphi_i\}_{i\in\N}\subset C^\infty(\Sigma)$ be an orthonormal basis of $L^2(\Sigma)$ such that $-\Delta_\Sigma\varphi_i=\lambda_i\varphi_i$ on $\Sigma$.

Consider the harmonic extensions $u_i=\mathcal{H}\varphi_i$ for $i=1,...,k$. They satisfy the following problem:
\begin{equation*}
\begin{cases}
\Delta u_i=0, & {\rm in\ }M,\\
u_i=\varphi_i, & {\rm on\ }\Sigma.
\end{cases}
\end{equation*}
Let $V\subset H^1(M)$ be the space spanned by $u_1,...u_k$. Any function $\phi\in V$ with $\int_{\Sigma}\phi^2dV_\Sigma=1$ can be written as $\phi=\sum_{i=1}^kc_iu_i$ with $\sum c_i^2=1$. Moreover $\Delta\phi=0$ for all $\phi\in V$. Using the min-max principle and  inequality \eqref{nor} together with the Cauchy-Schwarz inequality leads to
\begin{align*}
\sigma_k&\leq\sup_{\substack{0\ne \phi\in V\\\int_{\Sigma}\phi^2dV_\Sigma=1}}\int_{M}|\grad \phi|^2dV_M= \sup_{\substack{(c_1,...,c_k)\in\mathbb S^{k-1}}}\int_{M}\left|\grad\phi\right|^2dV_M\\
&\leq \max_{\substack{(c_1,...,c_k)\in\mathbb S^{k-1}}}\left(\int_{\Sigma} \left(\frac{\partial\phi}{\partial\n}\right)^2dV_\Sigma\right)^{\frac{1}{2}}\\
&\leq \constB+\left(\constB^2+\sup_{\substack{(c_1,...,c_k)\in\mathbb S^{k-1}}}\int_{\Sigma}\left|\grad_\Sigma\phi\right|^2\right)^{\frac{1}{2}}\\
&\leq {\constB}+\left(\constB^2+\sup_{\substack{(c_1,...,c_k)\in\mathbb S^{k-1}}}\sum_{i=1}^kc_i^2\lambda_i\right)^{\frac{1}{2}}
=\constB+\sqrt{\constB^2+\lambda_k}.
\end{align*}
This proves inequality \eqref{ineq:thm:mainstek}.

\noindent
\textbf{Part c}.
It remains to prove that $|\sigma_k-\sqrt{\lambda_k}|\leq\max\{\constA,2\constB\}$. Indeed,
\begin{align*}
\sigma_k-\sqrt{\lambda_k}&\leq \constB+\sqrt{\constB^2+\lambda_k}-\sqrt{\lambda_k}\\
&=\constB+\frac{\constB^2}{\sqrt{\constB^2+\lambda_k}+\sqrt{\lambda_k}}\leq 2\constB.
\end{align*}
Similarly,
\begin{align*}
\sqrt{\lambda_k}-\sigma_k&=\frac{\lambda_k-\sigma_k^2}{\sqrt{\lambda_k}+\sigma_k}
\leq \frac{\constA\sigma_k}{\sqrt{\lambda_k}+\sigma_k}\leq \constA.
\end{align*}
This implies
$|\sigma_k-\sqrt{\lambda_k}|\le\max\{\constA,2\constB\}.$
\end{proof}
The proof of Theorem \ref{thm:mainstek} is now a direct consequence.
\begin{proof}[Proof of Theorem \ref{thm:mainstek}]
Let $\Omega$ be a neighbourhood of the boundary $\Sigma$ and let $g_1,g_2$ be two Riemannian metrics which coincide on $\Omega$. Let $M_{\htd}$ be the admissible neighbourhood of $\Sigma$ for the metric $g_1$. Let $\hat{h}\in (0,\htd)$ be the largest number such that $M_{\hat{h}}\subset\Omega$. One can then define
$C>0$ using  formula \eqref{constant:CM}, for the metric $g_1$, with $\htd$ replaced by $\hat{h}$. Substituting $C$ for $\max\{\constA,2\constB\}$ in the proof of Theorem \ref{mainintro} leads to $|\sigma_k(g_i)-\sqrt{\lambda_k}|<C$ for $i=1,2$. This implies that $|\sigma_k(g_1)-\sigma_k(g_2)|\leq 2C$.
\end{proof}
\begin{proof}[Proof of Theorem \ref{rem:weakerH}]
The proof is identical to that of {\bf Part b} above, but we only need to use inequality \eqref{nor} under the assumption of Lemma \ref{estimatew/ricci} (see Remark \ref{remw/ricci}).
\end{proof}

\section{Signed curvature and convexity}\label{section:signedcurvature}
In this section we will give precise bounds on the constants $\constA$ and $\constB$ defined in \eqref{constant:CM} and prove Theorem \ref{thm:signedcurvature}.  In the situation where the sectional curvature is constrained to have a constant sign and where we impose convexity assumptions on the boundary $\Sigma$, we recover and improve the results of Xiong \cite{Xi}.

Our strategy is to estimate the quantities 
$(\htd-\delta)a(\delta)$, $(\htd-\delta)b(\delta)$ and $(\hb-\delta)\mu(\delta)$ which appear in the definitions \eqref{constant:cbm}  and \eqref{constant:scbm} of the constants $\cbm$ and $\scbm$, and thus also in the definitions \eqref{constant:CM} of $\constA$ and $\constB$. This is achieved by using the explicit formulas for $a(\delta)$, $b(\delta)$ and $\mu(\delta)$ provided by Lemma~\ref{lemma:riccationed} together with some technical lemmas regarding solutions of the Riccati equation in dimension $1$. These are stated in the next paragraph.

\subsection{Technical lemmas regarding Riccati equation in dimension 1}
In this paragraph we state a technical lemma which will be proved in the appendix. Its goal is to estimate the function
$f(x)=-(h-x)a(x)$ for solutions $a(x)$ of the Riccati equation.
\begin{lemma}\label{Lemma:technicalf(x)}
Let $K,\kappa\in\R$. Let $y:I\rightarrow\R$ be a solution to the Riccati equation
\begin{equation*}
y'+y^2+K=0, \quad y(0)=-\kappa,
\end{equation*}
with maximal interval $I$.  Let $h\in I$ and define the function $f:I\rightarrow\R$ by
$f(x)=-(h-x)y(x).$
The following inequalities hold:

\begin{center}
  \begin{tabular}{LCL}
  \mbox{\textbf{Condition on }}K&\mbox{\textbf{Condition on }}\kappa&\mbox{\textbf{Inequality}}\\
  \hline
K\geq 0&\kappa\in\R&\min\{0,h\kappa\}\leq f(x)\leq 1\\
K<0&\kappa\geq\sqrt{|K|}&0\leq f(x)\leq h\kappa
 \end{tabular}
\end{center}
\end{lemma}

\subsection{Totally geodesic boundary}

This corresponds to the parts (1) and (2) of Theorem \ref{thm:signedcurvature}.

\begin{proposition}\label{propo:totallygeodesicbdry}
Let $M$ be a compact manifold with nonempty totally geodesic boundary $\Sigma$. Assume that $|K|<\lambda^2$, where $\lambda>0$, on $M_{\hb}$. Then, \eqref{ineq:main1} and \eqref{ineq:thm:mainstek} hold with
$$\constA=\frac{n}{\tilde h}+\lambda\quad \mbox{and}\quad\constB =\frac{1}{2}\left(\frac{1}{\hb}+{n\lambda}\right),$$
where $\htd=\min\{\frac{\pi}{2\lambda},\hb\}.$

\smallskip

\noindent
Moreover, if $-\lambda^2<K<0$, then \eqref{ineq:main1} and \eqref{ineq:thm:mainstek} hold with
$$\constA=\frac{1}{\bar h}+\lambda\quad\mbox{and}\quad \constB=\frac{1}{2}\left(\frac{1}{\bar h}+n\lambda\right).$$
	\medskip
On the other hand, if $0<K<\lambda^2$, then \eqref{ineq:main1} and \eqref{ineq:thm:mainstek} hold with
$$\constA=n\max\{\bar h^{-1},\frac{2\lambda}{\pi}\},\quad\text{and}\quad \constB=\frac{1}{2\hb}.$$
\end{proposition}
\begin{proof}
It follows from Corollary \ref{coro:curvaturebounds} and Lemma \ref{lemma:riccationed} that
$$\htd=\min\{\roll(M),m(\lambda^2,0)\}=\min\{\frac{\pi}{2\lambda},\hb\}.$$
Moreover, the corresponding comparison functions are
\begin{align*}
\mu(\delta)=a(\delta)=\lambda\tanh(\lambda\delta)\qquad\mbox{ and }\qquad
b(\delta)=-\lambda\tan(\lambda \delta).
\end{align*}
It follows from Lemma \ref{Lemma:technicalf(x)} that $-(h-\delta)b(\delta)\leq 1$. Moreover, for every $0<\delta\le h\le \tilde h\le\hb$
$$(h-\delta)\lambda\tanh(\lambda\delta)<\lambda \tilde h.$$
Whence, 
$$\cbm(h)\le(n-1)+\lambda \tilde h\quad\mbox{and} \quad\scbm(h)\leq n\lambda \hb.$$
One can take
$$\constA=\frac{n}{\tilde h}+\lambda\quad \mbox{and}\quad\constB =\frac{1}{2}\left(\frac{1}{\hb}+{n\lambda}\right).$$

In the situation where $K<0$, 
$$\htd=\min\{\roll(M),m(0,0)\}=\hb.$$
The comparison functions are
\begin{align*}
\mu(\delta)=a(\delta)=\lambda\tanh(\lambda\delta)\qquad\mbox{ and }\qquad
b(\delta)=0.
\end{align*}
We have
$(h-\delta)a(\delta)\leq\hb\lambda$ and
$$\cbm\leq \hb\lambda\quad\mbox{and}\quad \scbm\le n\hb\lambda.$$ Thus one can take
$$\constA=\frac{1}{\bar h}+\lambda\quad\mbox{and}\quad \constB=\frac{1}{2}\left(\frac{1}{\bar h}+n\lambda\right).$$
In the case $K>0$ the comparison functions are
\begin{align*}
\mu(\delta)=a(\delta)=0\qquad b(\delta)=-\lambda\tan(\lambda\delta).
\end{align*} and as mentioned above $-(h-\delta)b(\delta)\le1$. Therefore, one can take
$$\constA=n\max\{\bar h^{-1},\frac{2\lambda}{\pi}\},\quad\text{and}\quad \constB=\frac{1}{2\hb}.$$
\end{proof}
Part (2) of Theorem \ref{thm:signedcurvature} follows immediately from the proof of Proposition \ref{propo:totallygeodesicbdry}.

\subsection{Horoconvex boundary} This corresponds to the part (3) of Theorem \ref{thm:signedcurvature}. Our hypothesis on the principal curvature implies local convexity as in \cite{Xi}, but the boundary of $M$ is not necessarily  connected.
\begin{proposition}\label{propo:horoconvexbdry}
Let $M$ be a compact manifold with nonempty boundary $\Sigma$. Assume that $-\lambda^2\leq K\leq 0$ on $M_{\hb}$ (where $\lambda>0$) and assume that each of the principal curvatures of $\Sigma$ satisfies
$\lambda<\kappa_i\leq\kappa_+$.
Then, one can take
$$\constA=n\max\{\kappa_+,\frac{1}{\hb}\}\qquad\mbox{and}\qquad\constB=\frac{1}{2\hb}.$$
\end{proposition}
\begin{proof}
It follows from Corollary \ref{coro:curvaturebounds} and Lemma \ref{lemma:riccationed} that
$$\htd=\min\{\roll(M),m(0,\kappa_+)\}=\min\{1/\kappa_+,\hb\}.$$
Moreover, the corresponding comparison functions satisfy
\begin{align*}
\mu(\delta)= a(\delta)=-\lambda \qquad\mbox{ and }\qquad
{-b(\delta)}=\frac{1}{-\delta+1/\kappa_+}.
\end{align*}
It follows that $(\htd-\delta)a(\delta)\leq 0$.
Moreover,
\begin{align*}
-(h-\delta)b(\delta)\leq \left(\frac{1}{\kappa_+}-\delta\right)\left(\frac{1}{-\delta+1/\kappa_+}\right)=1.
\end{align*} 
Whence, $\cbm\leq n-1$ and we can take $\scbm=0$ by Remark \ref{Bvalue0}. Therefore one can take
$$ \constA=\frac{n}{\htd}\qquad\mbox{and}\qquad \constB=\frac{1}{2\hb}.$$
\end{proof}

\subsection{Positive sectional curvature}

This corresponds to the part (4) of Theorem~\ref{thm:signedcurvature}.

\begin{proposition}\label{prop31}
Let $M$ be a compact manifold with nonempty boundary $\Sigma$. Let $\hb=\roll(M)$ be its rolling radius. Assume that $0<K<\lambda^2$, where $\lambda>0$, on $M_{\hb}$. Also assume that each principal curvatures satisfy  $0\leq\kappa_i\leq\kappa_+$ on $\Sigma$. 
Then, one can take
$$\constA=n\max\{\frac{1}{\hb},\sqrt{\lambda^2+\kappa_+^2}\},\quad\text{and}\quad \constB=\frac{1}{2\hb}.$$
\end{proposition}

\begin{proof}
It follows from Corollary \ref{coro:curvaturebounds} and Lemma \ref{lemma:riccationed} that
$$\htd=\min\{\roll(M),m(\lambda^2,\kappa_+)\}=\min\{\frac{1}{\lambda}(\frac{\pi}{2}-\arctan(\frac{\kappa_+}{\lambda})),\hb\}.$$
Let us show that
\begin{gather}\label{ineq:signedworkingpos}
\frac{\lambda}{\frac{\pi}{2}-\arctan(\frac{\kappa_+}{\lambda})}
\leq\sqrt{\lambda^2+\kappa_+^2}.
\end{gather}
Indeed this is equivalent to
\begin{gather*}
\frac{\lambda}{\sqrt{\lambda^2+\kappa_+^2}}
\leq\frac{\pi}{2}-\arctan(\frac{\kappa_+}{\lambda}),
\end{gather*}
Using that $\sin$ is increasing on $[0,\pi/2]$ this leads to the equivalent
\begin{gather*}
\sin(\frac{\lambda}{\sqrt{\lambda^2+\kappa_+^2}})
\leq\sin(\frac{\pi}{2}-\arctan(\frac{\kappa_+}{\lambda}))
=\cos(\arctan(\frac{\kappa_+}{\lambda})=\frac{\lambda}{\sqrt{\lambda^2+\kappa_+^2}},
\end{gather*}
which is true.

Moreover, the corresponding comparison functions satisfy
\begin{align*}
\mu(\delta)=a(\delta)=0
\mbox{ and }
-b(\delta)=\lambda\tan\left(\lambda\delta+\frac{1}{\lambda}\arctan(\frac{\kappa_+}{\lambda})\right)
\end{align*}

It follows from Lemma \ref{Lemma:technicalf(x)} that $-(h-\delta)b(\delta)\leq \kappa_+ h\leq 1$.
Whence, $\cbm\leq n-1$, $\scbm=0$, and one can take
$$\constA=n\max\{\frac{1}{\hb},\sqrt{\lambda^2+\kappa_+^2}\}\quad\text{ and }\quad \constB=\frac{1}{2\hb}.$$

\end{proof}

In some situations, the work of Donnelly and Lee \cite{DL91} can be used to compute $\bar h$ explicitly.
\begin{cor}\label{gXiong}
Let $M$ be a domain in a complete Riemannian manifold. Suppose that $\Sigma$ is connected and convex ($\kappa_->0$). Let $\lambda,\kappa_+>0$. Then 
\begin{enumerate}
\item Suppose that $\Sigma$ is totally geodesic.
If $0<K<\lambda^2$, then \eqref{ineq:main1} and \eqref{ineq:thm:mainstek} hold with
$$\constA={\frac{2n\lambda}{\pi}}\quad\text{and}\quad \constB=\frac{\lambda}{\pi}.$$
\item If $-\lambda^2\leq K\leq 0$ and each principal curvature satisfies $\lambda<\kappa_i<\kappa_+$, then \eqref{ineq:main1} and \eqref{ineq:thm:mainstek} hold with
$$\constA=n\kappa_+\qquad\mbox{and}\qquad{\constB=\frac{1}{2}\kappa_+}.$$
\item If $0<K<\lambda^2$ and each principal curvature satisfies $0<\kappa_i<\kappa_+$, then \eqref{ineq:main1} and \eqref{ineq:thm:mainstek} hold with
$$\constA=n\sqrt{\lambda^2+\kappa_+^2}\quad\text{and}\quad \constB= \frac{1}{2}\sqrt{\lambda^2+\kappa_+^2}.$$
\end{enumerate}
\end{cor}
\begin{proof}
The estimates on $\hb$ given in \cite[Theorem 3.11, Theorem 3.22]{DL91} are summarised in the following table.\\
$$\begin{tabular}{LLC}
 \multicolumn{2}{C}{\mbox{\bf Description}}& \bar h\\
  \hline
   |K|<\lambda^2,&0<\kappa_-<\lambda,&\frac{1}{\lambda}\min\left\{{\arctanh\left(\frac{\kappa_-}{\lambda}\right)}, {\arcot\left(\frac{\kappa_+}{\lambda}\right)}\right\}\\[.2cm]
   |K|<\lambda^2,&\kappa_-\ge\lambda,&\frac{\arcot\left(\frac{\kappa_+}{\lambda}\right)}{\lambda}\\[.2cm]
   -\lambda^2\leq K\leq 0,&0<\kappa_-<\lambda,&\min\left\{\frac{1}{\kappa_+}, \frac{\arcot\left(\frac{\kappa_+}{\lambda}\right)}{\lambda}\right\} \\[.2cm]
   -\lambda^2\leq K\leq 0,&\kappa_-\ge\lambda,&\frac{1}{\kappa_+} \\[.2cm]
   0<K<\lambda^2,&\kappa_+>\kappa_-\ge0&\frac{\arcot\left(\frac{\kappa_+}{\lambda}\right)}{\lambda}\\[.2cm]
   0<K<\lambda^2,&\kappa_-=\kappa_+=0,&\frac{\pi}{2\lambda}\\[.2cm]
 \end{tabular}$$
In each cases listed in the corollary, it is now enough to replace the value of $\bar h$ in Theorem~\ref{thm:signedcurvature} with the value given in the table above. To get part (3), we also use the following inequality which is proved  in the proof of Proposition \ref{prop31}:
$$
\frac{\lambda}{\arcot\left(\frac{\kappa_+}{\lambda}\right)}=\frac{\lambda}{\frac{\pi}{2}-\arctan(\frac{\kappa_+}{\lambda})}
\leq\sqrt{\lambda^2+\kappa_+^2}.
$$
\end{proof}

\subsection{Flat manifolds}\label{section:flatmanifolds} The goal is to prove Theorem \ref{thm:flatspace}.

\begin{proof}
\textbf{Case 1}: It follows from Corollary \ref{coro:curvaturebounds} and Lemma \ref{lemma:riccationed} that
$$\htd=\min\{\roll(M),m(0,\kappa_+)\}=\min\{1/\kappa_+,\hb\}.$$
Moreover, the corresponding comparison functions are 
\begin{gather*}
\mu(\delta)=a(\delta)=\frac{1}{\delta-1/\kappa_-},\mbox{ and } \quad
b(\delta)=\frac{1}{\delta-1/\kappa_+}.
\end{gather*}
Whence, for every $0<\delta< h$
\begin{align*}
 (h-\delta)a(\delta)=(h-\delta)\mu(\delta) \le h|\kappa_-|,\quad\mbox{and}\quad -(h-\delta)b(\delta)\leq 1. 
\end{align*} 
Thus, $\cbm\leq n-1+h|\kappa_-|$, and $\scbm=h|\kappa_-|$. Therefore, one can take
$$\constA=n\max\{\hb^{-1} \kappa_+\}+|\kappa_-|,\quad\text{and}\quad {\constB=\frac{1}{2}(\bar h^{-1}+n\vert \kappa_-\vert)}.$$
\textbf{Case 2}: $0\le \kappa_i<\kappa_+$. 
It follows from Corollary \ref{coro:curvaturebounds} and Lemma \ref{lemma:riccationed} that
$$\htd=\min\{\roll(M),m(0,\kappa_+)\}=\min\{1/\kappa_+,\hb\}.$$
Moreover, the corresponding comparison functions are 
\begin{gather*}
\mu(\delta)=a(\delta)=0,\mbox{ and } \quad
b(\delta)=\frac{1}{\delta-1/\kappa_+}.
\end{gather*}
Whence,
\begin{align*}
-(h-\delta)b(\delta)\leq 1.
\end{align*} 
Thus, $\cbm\leq n-1$, and $\scbm=0$. Therefore, one can take
$$\constA=\frac{n}{\htd}=n\max\{\hb^{-1}, \kappa_+\},\quad\text{and}\quad \constB=\frac{1}{2\hb}.$$

\medskip
\noindent
\textbf{Case 3}: $\kappa_-<\kappa\le 0$. It follows from Corollary \ref{coro:curvaturebounds} and Lemma \ref{lemma:riccationed} that
$$\htd=\min\{\roll(M),m(0,\kappa_-)\}=\hb.$$
Moreover, the corresponding comparison functions are 
\begin{gather*}
\mu(\delta)=a(\delta)=\frac{1}{\delta-1/\kappa_-},\quad\mbox{ and }\quad 
b(\delta)=0.
\end{gather*}
It follows that $(h-\delta)\frac{1}{\delta-1/\kappa_-} \le h|\kappa_-|$ for $0<\delta< h $, so that one can take
$$
\constA= \bar h^{-1}+\vert \kappa_-|\quad\mbox{and}\quad\constB=\frac{1}{2}(\bar h^{-1}+n\vert \kappa_-\vert).
$$
\end{proof}

\begin{rem}
	In all estimates above for $\constB$, one can consider $\kappa_-$ as a lower bound for the mean curvature of $\Sigma$ and $\alpha$ as a lower bound on the Ricci curvature.  The same estimates remain true. See also Theorem \ref{rem:weakerH}.
\end{rem}

\subsection{Cylindrical boundaries} \label{cylindricalcase}
 Let assume that  a neighbourhood of the boundary $\Sigma$ is isometric to $\Sigma\times[0,L)$. In the Riccati equation \eqref{riequ1} only the sectional curvature along the normal geodesic appears, and it vanishes in the present example. Because moreover the boundary $\Sigma$ is totally geodesic, we have $\scbm=\cbm=0$ and $\roll(M)\ge L$.
Therefore, \eqref{ineq:main1} and \eqref{ineq:thm:mainstek} hold with $2\constB=\constA=\frac{1}{L}$.
In particular $|\sigma_j-\sqrt{\lambda_j}|<1/L.$
In fact, using the min-max characterization of eigenvalues and separation of variables, it is possible to show that for each $j$ which is larger than the number of connected components of $\Sigma$, one has
$$\sqrt{\lambda_j}\tanh(\sqrt{\lambda_j}L)\leq\sigma_j\leq\sqrt{\lambda_j}\coth(\sqrt{\lambda_j}L).$$
This implies that $\sigma_j-\sqrt{\lambda_j}$ tends to zero very fast with $\lambda_j\to\infty$.

\section{Examples and remarks} \label{section:examples}
In this section, we discuss the necessity of the hypothesis of Theorem \ref{mainintro} and give different kinds of examples to illustrate this.

\begin{example} \label{rolling}
The condition on the rolling radius is clearly a necessary condition. An easy example, with two boundary components, is given as follow: Take $M=T^n \times [0,L]$ where $T^n$ is a n-dimensional flat torus. The sectional curvature of $M$ is $0$, the principal curvatures of $\partial M=T^n$ are $0$. As $L\to 0$, the rolling radius tends to $0$, and $\sigma_k \to 0$ for all k (\cite{CEG11}). However $\lambda_3 (\partial M)$ is fixed and strictly positive. This contradicts Inequality (\ref{ineq:main1}), and shows that we need a lower bound on the rolling radius.

We can construct an example with one boundary component in the spirit of \cite[Example 2]{CGI13}. The length of the boundary is bounded, and this insures that $\lambda_2$ is bounded away from $0$. The curvature of the boundary is bounded, but as the boundary becomes ``close to itself", this allow to construct small $\sigma_2$.  
\begin{figure}[ht!]  
\begin{center}
	\includegraphics[width=7cm]{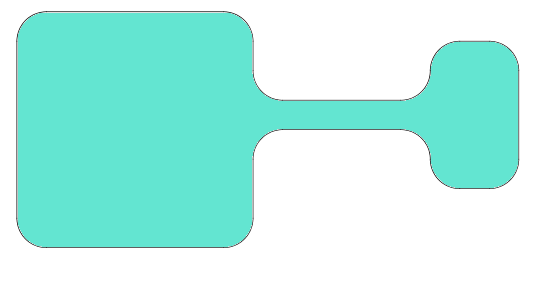}
\end{center}
 \end{figure}
\end{example}

\begin{example} \label{curvature}
The condition on the curvature of the ambient space is also a necessary condition: in Theorem 3.5 of \cite{CEG17}, we have constructed an example where the boundary is fixed (which implies that all $\lambda_i$ are fixed), its principal curvatures equal to $0$ and the rolling radius uniformly bounded from below. However, $\sigma_2$ becomes arbitrarily large. This contradicts Inequality (\ref{ineq:thm:mainstek}) and this comes form the fact that the curvature of the ambient space goes to $\infty$. 
 
	\medskip
	We can also adapt this example in order to show that with{out} constraint on the curvature, $\sigma_3$ may be small and not $\lambda_3$, even with a control of the rolling radius. We take $\Sigma$ of dimension $n\ge 3$, and consider a Riemannian metric $g_\varepsilon=
  h_\varepsilon^2(t)g_0+dt^2$ which coincides with $g$ in a $\epsilon$-neighbourhood of $\Sigma$ in $M$  and such 
  that $ h_\varepsilon^2(t)$ is $\le \epsilon$ outside a neighbourhood of radius $2\epsilon$ of $\Sigma$. A direct and simple calculation shows that if $f$ is a function of norm $1$ on $\Sigma$, then the Rayleigh quotient of $f$ viewed as a function on $M$ (that is $F(p,t)=f(p)$) has a Rayleigh quotient converging to $0$ with $\epsilon$. All the Steklov spectrum of $(M,g_\varepsilon)$ goes to $0$ with $\epsilon$.
   
\end{example}
 
\begin{example} \label{constant}
One can use Example 6.3 given in the proof of theorem 7 in \cite{CGR} in order to construct an example with $\sigma_2$ bounded from below by a positive constant, but $\lambda_2$ arbitrarily small. Example 6.3 consist in gluing together a finite number of fundamental pieces, so that we get a family of example with uniformly bounded geometry (curvature, curvature of the boundary, rolling radius). Moreover, we have shown that $\sigma_2$ was uniformly bounded below. At the same time, the boundary is a curve of length going to $\infty$, so that $\lambda_2$ goes to $0$.  
\end{example}

\section{Appendix: the Riccati equation in dimension one}
Let $K,\kappa\in\R$. Let $\mathbb{M}_{K}$ be the space-form of
constant sectional curvature $K$. Let
$\Omega\subset\mathbb{M}_{K}$ be an open connected set with connected smooth umbilical boundary
$\Sigma=\partial\Omega\subset\mathbb{M}_\alpha$ with principal curvature $\kappa$.
We recall that, in this situation, we have to investigate equation \eqref{1driccati}:
$$\begin{cases}
  y'+y^2+K=0,\\
  y(0)=-\kappa.
\end{cases}$$
Our goal is to compute the solution $a(s)$ explicitly and to
determine the maximal $m>0$ such that the solution $y$ exists on $[0,m)$ in order to prove Lemma \ref{lemma:riccationed} and \ref{lemma:rhonsmallerone}.

\medskip
\noindent

\begin{proof}[\textbf{Proof of Lemma \ref{lemma:riccationed}.}]

In the particular situation where $\kappa ^2+K=0$, the solution
is the constant function $y:\R\rightarrow\R$ defined by
$y(t)=-\kappa$.
In all other situations, the solution satisfies
\begin{gather*}
-\frac{y'}{y^2+K}=1\qquad\mbox{ on }\,I.
\end{gather*}
Integration from $0$ to $x\in I$ leads to
\begin{gather}\label{eq:ricaonedint}
-\int_0^x\frac{y'(t)}{y^2(t)+K}\,dt=x.
\end{gather}
The explicit representation of this integral by elementary functions depend on the sign of the curvature $K$.

\subsubsection*{Case $K=0$}
It follows from
\begin{align*}
x=-\int_0^x\frac{y'(t)}{y^2(t)}\,dt=\frac{1}{y(x)}-\frac{1}{y(0)}=\frac{1}{y(x)}+\frac{1}{\kappa}
\end{align*}
that
$$y(x)=\frac{1}{x-\frac{1}{\kappa}}=\frac{\kappa}{\kappa x-1}$$
The maximal existence time therefore is
$$m(0,\kappa)=
\begin{cases}
  \frac{1}{\kappa}&\mbox{ if }\kappa>0,\\
  \infty&\mbox{ if }\kappa\leq 0.
\end{cases}$$

\subsubsection*{Case $K>0$} Let $\lambda>0$ be such that $K=\lambda^2$, and observe that
\begin{align*}
-\int_0^x\frac{y'(t)}{y^2(t)+K}\,dt&=-\frac{1}{\lambda}\int_0^x\frac{\frac{y'(t)}{\lambda}}{(\frac{y(t)}{\lambda})^2+1}\,dt\\
&=-\frac{1}{\lambda}\left(\arctan(\frac{y(x)}{\lambda})-\arctan(\frac{y(0)}{\lambda})\right)\\
&=\frac{1}{\lambda}\left(-\arctan(\frac{y(x)}{\lambda})-\arctan(\frac{\kappa}{\lambda})\right)
\end{align*}
Together with \eqref{eq:ricaonedint} this implies that
$$y(x)=-\lambda\tan\left(\lambda x+\arctan(\frac{\kappa}{\lambda})\right)$$
For $x>0$, this solution is well-defined as long as
$$\lambda x+\arctan(\frac{\kappa}{\lambda})<\pi/2.$$
That is,
$$m(K,\kappa)=\frac{1}{\lambda}\left(\frac{\pi}{2}-\arctan(\frac{\kappa}{\lambda})\right).$$

\subsubsection*{Case $K<0$} 
Let $\lambda>0$ be such that
$K=-\lambda^2$.
The left-hand-side of \eqref{eq:ricaonedint} is
\begin{align*}
  -\int_0^x\frac{y'(t)}{y^2(t)+K}\,dt
  &=
  \frac{1}{\lambda}\int_0^x\frac{\frac{y'(t)}{\lambda}}{1-(\frac{y(t)}{\lambda})^2}\,dt\\
  &=
  \frac{1}{\lambda}\int_{y(0)/\lambda}^{y(x)/\lambda}\frac{1}{1-u^2}\,du
  =\frac{1}{\lambda}\int_{-\kappa/\lambda}^{y(x)/\lambda}\frac{1}{1-u^2}\,du.
\end{align*}
Now if $-\kappa/\lambda\in (-1,1)$ then
$$\frac{1}{\lambda}\int_{-\kappa/\lambda}^{y(x)/\lambda}\frac{1}{1-u^2}\,du=\frac{1}{\lambda}\left(\arctanh(\frac{y(x)}{\lambda})-\arctanh(-\frac{\kappa}{\lambda})\right).$$
On the other hand, if $|\kappa/\lambda|>1$ then
$$\frac{1}{\lambda}\int_{-\kappa/\lambda}^{y(x)/\lambda}\frac{1}{1-u^2}\,du=\frac{1}{\lambda}\left(\arcoth(\frac{y(x)}{\lambda})-\arcoth(-\frac{\kappa}{\lambda})\right).$$
It follows that
\begin{gather}
  y(x)=
  \begin{cases}
    \lambda\tanh\left(\lambda
    x-\arctanh(\frac{\kappa}{\lambda})\right)&\mbox{ if }
    |\kappa|<\lambda;\\
    \lambda\coth\left(\lambda
    x-\arcoth(\frac{\kappa}{\lambda})\right)&\mbox{ if }
    |\kappa|>\lambda.
  \end{cases}
\end{gather}

The maximal time of existence is
\begin{gather*}
  m(y,\kappa)=
  \begin{cases}
    \frac{1}{\lambda}\arcoth(\frac{\kappa}{\lambda})&\mbox{ if }\kappa>\lambda\\
    +\infty&\mbox{ if }|\kappa|<\lambda,\\
    +\infty&\mbox{ if }\kappa<-\lambda,\\
  \end{cases}
\end{gather*}

\end{proof}

\medskip

We can now prove Lemma \ref{lemma:rhonsmallerone} and Lemma \ref{Lemma:technicalf(x)} which provide bounds on $f(x)=-(h-x)y(x)$.

\begin{proof}[\textbf{Proof of Lemma \ref{lemma:rhonsmallerone}}]

  We use the formula for $y(s)$ given in Lemma~\ref{lemma:riccationed} to
  verify the statement in different  cases. The proof is a
  straightforward calculation.  \\ 
  \textit{Case 1.} $K=\lambda^2, \lambda>0$. Here we use the fact that 
  \begin{equation}\label{cot}x\cot(x)\le 1,\quad \forall x\in
    [0,{\pi}),\end{equation}  
  with equality at $x=0$.\\
 Recall from Lemma \ref{lemma:riccationed} that for $K=\lambda^2$ we have 
  $$-y(s)=\lambda\tan\left(\lambda   s+\arctan(\frac{\kappa}{\lambda})\right)=\lambda\cot\left(\arcot(\frac{\kappa}{\lambda})-\lambda   s\right)$$ for any $0\le s\le
  \frac{1}{\lambda}{\arcot(\frac{\kappa}{\lambda})}$. We used the identity $\arctan(x)+\arcot(x)=\frac{\pi}{2}$. Thus
  \[-(h-s)y(s)\le \left(\arcot(\frac{\kappa}{\lambda})-\lambda s\right) \cot\left(\arcot(\frac{\kappa}{\lambda})-\lambda s\right)\le1.\]

\noindent \textit{Case 2.} $K=0$.  If $\kappa\le0$, then $-y(s)\le 0$ on $I$. If $\kappa>0$, then $-(h-s)y(s)=(h-s)\frac{\kappa}{1-\kappa s}\le 1$ if and only if $h<\frac{1}{\kappa}$. Hence, it is clear that in both situations the statement of the lemma holds. \smallskip
\end{proof}

\begin{proof}[\textbf{Proof of Lemma \ref{Lemma:technicalf(x)}}] We consider two cases separately.
	
\smallskip

\noindent\textbf{Case 1}: $K\ge 0$.
From Lemma \ref{lemma:rhonsmallerone}, we know that $f(x)\le1$ for $K\ge0$. It remains to show that $f(x)\ge\min\{0,h\kappa\}.$
\smallskip

\noindent\textbf{Case 1(a).~} $K\textsl{}=0$. If $\kappa=0$, then $y\equiv 0$ and $f\equiv 0$ and there is nothing more to prove. If $\kappa>0$, then $I=(0,1/\kappa)$ and $f(x)=-(h-x)\frac{1}{x-1/\kappa}>0$.
 If $\kappa<0$, then $I=(0,\infty)$ and $f(x)=-(h-x)\frac{1}{x-1/\kappa}$. It follows from $(h-x)\le h$ and $(x-\frac{1}{\kappa}) \ge \frac{1}{\vert \kappa \vert}$ that
$$
 -(h-x)\frac{1}{x-1/\kappa}\ge h\kappa.
$$
\textbf{Case 1(b).~} $K>0$.

From Lemma \ref{lemma:rhonsmallerone} and its proof we know
$$f(x)=\lambda(h-s)\cot\left(\arcot(\frac{\kappa}{\lambda})-\lambda   s\right),\quad\text{with}~ I=\left(0, \frac{1}{\lambda}{\arcot(\frac{\kappa}{\lambda})}\right).$$
If $\arcot(\frac{\kappa}{\lambda})-\lambda   s\le\frac{\pi}{2}$, then $f(x)\ge0$. If $\arcot(\frac{\kappa}{\lambda})-\lambda   s\ge\frac{\pi}{2}$, then $\cot\left(\arcot(\frac{\kappa}{\lambda})-\lambda   s\right)\le0$ and
$$f'(x)=-\lambda \cot\left(\arcot(\frac{\kappa}{\lambda})-\lambda   s\right)+\frac{\lambda^2(h-s)}{\sin^2(\arcot(\frac{\kappa}{\lambda})-\lambda   s)}\ge0.$$
Therefore, $f(x)\geq f(0)=h\kappa$.

\smallskip

\noindent  
\textbf{Case 2}: $K<0$ and $\kappa\geq \lambda:=\sqrt{|K|}$.

The maximal existence interval is $I=(0,\frac{1}{\lambda}\arcoth(\kappa/\lambda))$ and
$$f(x)=-(h-x)\lambda\coth(\lambda x-\arcoth(\kappa/\lambda)).$$
For any $\gamma\geq 0$ the function $\psi_\gamma:\R_+\rightarrow\R$ defined by
$$\psi_\gamma(t)=t\coth(t+\gamma)$$
is increasing.
Let $\phi=\frac{1}{\lambda}\arcoth(\kappa/\lambda).$
Observe that for $\gamma=\lambda\phi-\lambda h$
$$0\leq f(x)=\psi_\gamma(\lambda h-\lambda x).$$
It follows that
\begin{align*}
  0\leq f(x)&=(\lambda h-\lambda x)\coth(\lambda\phi-\lambda x)\\
  &=(\lambda h-\lambda x)\coth(\lambda h-\lambda x+\lambda\phi-\lambda h)
  \leq \lambda h\coth(\lambda\phi)=h\kappa.
\end{align*}
\end{proof}

\section*{Acknowledgment}
Part of this work was done while BC was visiting Universit\'e Laval. He thanks the personnel from the D\'epartement de math\'ematiques et de statistique for providing good working conditions.
Part of this work was done while AG was visiting Neuch\^atel. The support of the Institut de Math\'ematiques de Neuch\^atel is warmly acknowledged.
AH would like to thank the Mittag--Leffler Institute for the support and  for an excellent working condition during the starting phase of the project. 

\bibliographystyle{plain}
\bibliography{refstek}

\end{document}